\documentclass[11pt]{amsart}
\usepackage{tkz-euclide}
\usepackage{latexsym}
\usepackage{setspace}
\usepackage{cases}
\usepackage{amssymb,amsmath,mathrsfs}
\usepackage{amsthm}
\usepackage{hyperref}
\usepackage{enumitem}
\usepackage[capitalise]{cleveref}
\usepackage{mathtools}
\usepackage{mathabx}



\allowdisplaybreaks 
\newtheorem{theorem}{Theorem}[section]
\newtheorem{definition}{Definition}[section]
\newtheorem{prop}[theorem]{Proposition}

\newtheorem{lemma}[theorem]{Lemma}

\newcommand{\pert}{\nabla(\W\Lq)}
\newcommand{\fkl}{f_{klk'l'}}
\newcommand{\pertperp}{\grad^\perp(\W\Lq)}
\newcommand{\Dtq}{D_{t,q}}
\newcommand{\Lq}{L_{q+1}}
\newcommand{\dq}{\delta_q}
\newcommand{\dqone}{\delta_{q+1}}
\newcommand{\dqtwo}{\delta_{q+2}}
\newcommand{\laq}{\lambda_q}
\newcommand{\laqone}{\lambda_{q+1}}
\newcommand{\muqone}{\mu_{q+1}}
\newcommand{\wkl}{w_{kl}}
\newcommand{\wklprime}{w_{k'l'}}
\newcommand{\Xl}{\mathcal{X}_{l}}
\newcommand{\Xlprime}{\mathcal{X}_{l'}}
\newcommand{\M}{\mathring{M}}  

\newcommand{\E}{\mathcal{E}}
\newcommand{\X}{\mathbb{X}}
\newcommand{\bark}{\bar{k}}
\newcommand{\V}{\mathbb{V}}
\newcommand{\W}{\mathbb{W}_{q+1}}
\newcommand{\Wqone}{\mathbb{W}_{q+1}}
\newcommand{\D}{\mathcal{D}}
\newcommand{\akl}{a_{kl}}
\newcommand{\aklprime}{a_{k'l'}}
\newcommand{\aklexp}{\akl e^{i\laqone(\Phi_l-x)\cdot k}}
\newcommand{\aminusklexp}{a_{-kl} e^{i\laqone(\Phi_l-x)\cdot k}}
\newcommand{\aklexpprime}{\aklprime e^{i\laqone(\Phi_{l'}-x)\cdot k'}}
\newcommand{\expk}{e^{i\laqone k\cdot x}}
\newcommand{\expkprime}{e^{i\laqone k'\cdot x}}
\newcommand{\torus}{\mathbb{T}^3}
\newcommand{\Bdot}{\mathring{B}}

\newcommand{\Mdot}{\mathring{M}}
\newcommand{\Cbar}{\bar{C}}

\newcommand{\inverselap}{(-\overline{\Delta})^{-1}}

\newcommand{\grad}{\overline{\nabla}}

\newcommand{\lap}{\overline{\Delta}}

\newcommand{\kbar}{\overline{k}}

\newcommand{\Pgrad}{\mathbb{P}_\nabla}

\newcommand{\Pbarqone}{\bar{{\mathbb{P}}}_{\approx \laqone}}
\newcommand{\Pgradk}{{\mathbb{P}}^\nabla_{q+1,k}}

\newcommand{\Pgradminusk}{{\mathbb{P}}^\nabla_{q+1,-k}}
\newcommand{\Pgradkprime}{{\mathbb{P}}^\nabla_{q+1,k'}}
\newcommand{\Pgradkbarperp}{\overline{\Pgradk}^\perp}
\newcommand{\Pgradkbarperpprime}{\overline{\Pgradkprime}^\perp}
\newcommand{\Pcurl}{\mathbb{P}_{\curl}}

\newcommand{\Pgradbar}{\mathbb{P}^{\grad}}
\newcommand{\Pgradbarperp}{\mathbb{P}^{\grad^\perp}}
\newcommand{\kbarperpprime}{{\kbar'}^\perp}
\newcommand{\gradperpinverse}{\left(\grad^\perp\right)^{-1}}

\def\XXint#1#2#3{{\setbox0=\hbox{$#1{#2#3}{\int}$ }
\vcenter{\hbox{$#2#3$ }}\kern-.6\wd0}}

\newcommand{\Id}{\operatorname{Id}}

\newcommand{\curl}{\operatorname{curl}}

\newcommand{\supp}{\operatorname{supp}}
\newcommand{\Riesz}{\mathcal{R}}

\numberwithin{equation}{section}

\title[Non-uniqueness of Weak Solutions to the 3D Quasi-Geostrophic Equations]{Non-uniqueness of Weak Solutions to the 3D Quasi-Geostrophic Equations}

\begin{document}
\author{Matthew D. Novack}
\begin{abstract}
    We show that weak solutions to the 3D quasi-geostrophic system in the class $C^\zeta_{t,x}$ for $\zeta<\frac{1}{5}$ are not unique and may achieve any smooth, non-negative energy profile. Our proof follows a convex integration scheme which utilizes in a crucial way the stratified nature of the quasi-geostrophic velocity field, providing a link with the 2D Euler equations. In fact we observe that under particular circumstances our construction coincides with the convex integration scheme for the 2D Euler equations introduced by Choffrut, De Lellis, and Szek\'{e}lyhidi \cite{cdlsj12} and recovers a result which can already be inferred from the arguments of Buckmaster, De Lellis, Isett, and Sz\'{e}kelyhidi \cite{bdlisj15} or Buckmaster, Shkoller, and Vicol \cite{bsv16}.
\end{abstract}
\maketitle
\tableofcontents
\section{Introduction}

The inviscid three-dimensional quasi-geostrophic equation, or 3D QG, is a system of equations used to describe oceanic and atmospheric circulation. In this paper, we pose the equations for $(t,x,y,z)\in\mathbb{R} \times \mathbb{T}^2 \times  [0,2\pi]$, which corresponds to the physical setting of a stratified, rotating fluid with solid walls at $z=0$ and $z=2\pi$ \cite{dg}. The velocity field is given in terms of the stream function $\Psi:\mathbb{R}\times \mathbb{T}^2 \times [0,2\pi]\rightarrow \mathbb{R}$ by
$$ \grad^\perp\Psi(t,x,y,z) := (-\partial_y\Psi(t,x,y,z), \partial_x\Psi(t,x,y,z), 0). $$
We use the notation $\partial_\nu \Psi$ to denoted the outward pointing normal derivatives of $\Psi$ at $z=0,2\pi$. The following set of equations then governs the evolution of $\nabla\Psi$:

\[\begin{dcases}
   \partial_t(\Delta\Psi) + \grad^\perp\Psi \cdot \nabla(\Delta\Psi) = 0 & (t,x,y,z)\in \mathbb{R}\times\mathbb{T}^2 \times [0,2\pi] \\
   \partial_t(\partial_\nu\Psi) + \grad^\perp\Psi \cdot \nabla(\partial_\nu\Psi) = 0 & (t,x,y,z)\in\mathbb{R}\times\mathbb{T}^2\times\{0,2\pi\}.
\end{dcases}
\]
The above equations state that the potential vorticity $\Delta\Psi$ and Neumann derivative $\partial_\nu \Psi$ are transported by the velocity field $\grad^\perp\Psi$. In this paper, we shall exclusively use the following reformulation due to Puel and Vasseur \cite{pv}.
\[\begin{dcases}
   \partial_t(\nabla\Psi) + \grad^\perp\Psi \cdot \nabla(\nabla\Psi) = \curl(Q) & (t,x,y,z)\in \mathbb{R}\times\torus\\
   \curl(Q)\cdot(0,0,1) = 0 & (t,x,y,z)\in\mathbb{R}\times\mathbb{T}^2\times\{0,2\pi\}.
\end{dcases}
\]
Weak solutions to the reformulated problem are defined via the following equality for all test functions $\phi$ in $C^{\infty}\left(\mathbb{R}\times\torus\right)$ which are compactly supported in time (\cite{pv}):
$$ \int_{\mathbb{R}}\int_{\torus} \partial_t(\nabla\phi)\cdot\nabla\Psi + \nabla\Psi \cdot \left( \grad^\perp\Psi\cdot\nabla\nabla\phi \right) \,dt\,dx\,dy\,dz = 0. $$

The original equations should be understood as analogous to the vorticity formulation of the Euler equations, while the reformulation corresponds to the Euler equations in their standard form. As the purpose of this paper is a non-uniqueness theorem, it is then natural that the weak solutions we construct will be defined at the level of the reformulation. Under sufficient integrability assumptions on $\Delta\Psi$ and $\partial_\nu\Psi$ (not satisfied by the solutions we construct in this paper), it is shown in \cite{pv} and \cite{novackweak} that weak solutions to the reformulated problem are weak solutions to the original system of equations, and vice versa.
The vector field $\curl(Q)$ plays a role analogous to that of the pressure in the Euler equations and is therefore defined in terms of a projection operator applied to the nonlinear term $\grad^\perp\Psi\cdot\nabla(\nabla\Psi)$, with $Q$ itself solving the elliptic equation
$$  -\Delta Q = \curl\left( \grad^\perp\Psi \cdot \nabla(\nabla\Psi) \right).  $$
Since weak solutions are defined via integration against vector fields $\nabla\phi$, $\curl(Q)$ does not appear in the weak formulation. In this paper, we prove the first result demonstrating anomalous behavior of the energy profile for such weak solutions (in fact for any type of weak solutions) to 3D QG.
\begin{theorem}\label{maintheorem}
Let $e:\mathbb{R}\rightarrow[0,\infty)$ be a smooth, compactly supported function and $\zeta\in \left(0,\frac{1}{5}\right)$.  Then there exist vector fields $\nabla\Psi\in C^\zeta\left(\mathbb{R}\times\torus\right)$ and $Q\in L^\infty\left(\mathbb{R};C^{2\zeta}(\torus)\right)$ such that $\nabla\Psi$ is a weak solution to 3D QG and
$$ \int_{\mathbb{T}^3} |\nabla\Psi(t,x,y,z)|^2 \,dx\,dy\,dz = e(t). $$
\end{theorem}
The proof of \cref{maintheorem} proceeds via a convex integration scheme.  While we shall postpone a more detailed description of the proof for the time being, we pause to make two remarks. First, the methods used in the proof of \cref{maintheorem} can be adapted with minimal effort to demonstrate the existence of infinitely many weak solutions sharing the same smooth initial data.  We outline the adjustments needed to prove such a statement in subsection \cref{ss:infinitelymany} in the appendix. Such pathological behavior has been investigated using convex integration methods in the literature, and we refer to the works of De Lellis and Sz\'{e}kelyhidi \cite{deLellis2008}, Sz\'{e}kelyhidi and Wiedemann \cite{Szkelyhidi2012}, Daneri \cite{Daneri2014}, Daneri and Sz\'{e}kelyhidi \cite{Daneri2017}, Colombo, De Lellis, and De Rosa \cite{Colombo2018}, and De Rosa \cite{DeRosa2018} for further discussion of this question and related phenomena. 

Secondly, we point out that the stratification provides a link between 3D QG and the two-dimensional Euler equations.  Convex integration for 2D Euler equations was first considered by Choffrut, De Lellis and Sz\'{e}kelyhidi in the class of continuous solutions \cite{cdlsj12}.  In \cite{bsv16}, Buckmaster, Shkoller, and Vicol observe that by replacing the Beltrami waves used in \cite{bdlisj15} with Beltrami plane waves as in \cite{cdlsj12}, the following theorem can be shown using either the methods from their paper or those of Buckmaster, De Lellis, Isett and Sz\'{e}kelyhidi \cite{bdlisj15}.
\begin{theorem}\label{maintheorem2d}
Consider the two-dimensional Euler equations
\[\begin{dcases}
   \partial_t u + u \cdot \nabla u + \nabla p = 0 & (t,x,y)\in\mathbb{R}\times\mathbb{T}^2\\
   \nabla \cdot u = 0 & (t,x,y)\in\mathbb{R}\times\mathbb{T}^2.
\end{dcases}
\]
Given a smooth, compactly supported energy profile $e:\mathbb{R}\rightarrow[0,\infty)$ and $\zeta\in\left(0,\frac{1}{5}\right)$, there exists $(u,p)$ which solves the equations in the sense of distributions with $u\in C^{\zeta}(\mathbb{R}\times\mathbb{T}^2)$, $p\in L^\infty\left(\mathbb{R};C^{2\zeta}(\torus)\right)$, and
$$ \int_{\mathbb{T}^2} |u(t,x,y)|^2 \,dx\,dy = e(t). $$
\end{theorem}
We demonstrate that \cref{maintheorem2d} follows as well from our proof of \cref{maintheorem} to emphasize the connection between the 2D Euler equations and three-dimensional, stratified, rotating fluids.

\subsection{An Outline of the Scheme}
While the reader will likely recognize in our proof shared attributes with recent convex integration schemes, 3D QG presents several particular difficulties. Therefore, in this section we provide a road map for our argument, along the way highlighting the new aspects of our construction and introducing some terminology and notation we will use throughout the paper which is specific to 3D QG.

\subsubsection{Handling the Solid Walls and Finding Stationary Solutions}

An immediate obstacle to adapting standard convex integration schemes to 3D QG is the presence of physical boundary conditions at the solid walls $z=0$ and $z=2\pi$. To our knowledge, a significant majority of the existing arguments, with the exception of a work of Isett and Oh \cite{isettoh2016} which we discuss in \cref{background}, apply to systems with periodic boundary conditions.  At the level of the reformulation, the physical boundary conditions are manifest in the requirement that the third component of $\curl(Q)$ vanishes at $z=0$ and $z=2\pi$. We are able to circumvent this issue by building solutions which vanish in a neighborhood of $z=0$ and $z=2\pi$.  A convenient upshot of this approach is that the solutions we construct are in fact periodic in $z$ as well, allowing access to the Fourier analytic tools which have become commonplace in convex integration arguments.  

Let us now explain why we handled the solid boundaries in this manner by searching for a stationary solution to 3D QG which can serve as a building block for a convex integration scheme. Throughout the ensuing discussion and the rest of the paper, we identify $\nabla\W\otimes\grad^\perp\W$ with a matrix whose rows are specified by the components of $\nabla\W$ and whose columns are specified by the components of $\grad^\perp \W$.  Differential operators with a bar such as $\grad \cdot$ include derivatives in $x$ and $y$ only.  For example, the divergence $\grad\cdot$ of the above matrix is taken row by row and differentiates in $x$ and $y$ only, thus ignoring the third column (which is already zero).  We seek $\nabla\W$ which solves 
\[\begin{dcases}
    \grad \cdot \left( \nabla \W \otimes \grad^\perp \W \right) = \curl(Q) & (x,y,z)\in \mathbb{T}^2\times[0,2\pi]\\
   \curl(Q)\cdot(0,0,1) = 0 & (x,y,z)\in\mathbb{T}^2\times\{0,2\pi\}.
\end{dcases}
\]
A straightforward calculation (\cref{stationarysolutions}) shows that eigenfunctions of the Laplacian will satisfy the first equation. Therefore, a natural candidate for a stationary solution would be a linear combination of $\mathbb{T}^3$-periodic complex exponentials $e^{i k\cdot x}$ supported on a sphere on frequency. The second equation, however, requires specific behavior of the eigenfunctions at the boundary. A natural way to achieve this would be to impose that $\partial_z \W$ vanishes at $z=0, 2\pi$. Therefore, if the Fourier series of $\W$ contains the term $c_k e^{i(k_1,k_2,k_3)\cdot (x,y,z)}$, it should also contain the term $c_k e^{i(k_1,k_2,-k_3)\cdot (x,y,z)}$.  Then
$$ \partial_z \W = c_k e^{i(k_1,k_2,0)\cdot(x,y,0)} ik_3 \left( e^{i k_3 z} - e^{-i k_3 \cdot z} \right)  $$
will vanish at $z=0,2\pi$. The reader may recall that the cancellation in convex integration schemes is predicated on building blocks for which the mean of the tensor product of high-frequency blocks cancels low-frequency errors. This means that the Fourier series of $\W$ must additionally contain the terms $\bar{c_k}e^{-i(k_1,k_2,k_3)\cdot (x,y,z)}$ and $\bar{c_k}e^{-i(k_1,k_2,-k_3)\cdot (x,y,z)}$.  However, this has the effect of annihilating the low frequency portion of $\grad^\perp \W \partial_z \W$ in the entirety of $\torus$. Indeed, choosing modes 
$$\left(k_1, k_2, k_3\right), \quad \left(-k_1, -k_2, -k_3\right), \quad \left(k_1, k_2, -k_3\right), \quad \left(-k_1, -k_2, k_3\right),$$
denoting $\bark^\perp = (-k_2, k_1, 0)$, and writing out the low frequency portion of the third row of $\nabla\W \otimes \grad^\perp \W$, we obtain
$$ |c_k|^2 (k_3) \bark^\perp + |c_k|^2(-k_3)\bark^\perp =0. $$
So we should find another way to enforce the boundary conditions at $z=0$ and $z=2\pi$.

Towards the goal of producing stationary solutions, we instead introduce a cutoff function $\Lq$ which depends on $z$ only and consider the function $\nabla \left(\W \Lq\right)$.  The viability of the cutoff function is visible in the equality
\begin{align*}
\grad \cdot \left( \nabla\left( \Lq \W \right) \otimes \grad^\perp \left( \Lq \W \right) \right) &= \Lq^2 \grad \cdot \left( \nabla \W \otimes \grad^\perp \W \right)\\
&= \Lq^2 \curl(Q)\\
&= \curl(\Lq^2 Q) - \textnormal{lower order terms}
\end{align*}
We prove and discuss this equality in \cref{stationarysolutions} and \cref{importantcancellation}, with the basic idea being that we have constructed solutions \textit{which are stationary to leading order}. 

\subsubsection{The Iterative Scheme}
With appropriate building blocks in hand, we build a solution $\nabla\Psi$ through an iterative process which specifies the behavior of $\nabla\Psi$ at higher and higher frequencies in each subsequent stage.  After $q$ stages of this iteration, we have vector fields $\nabla\Psi_q$, $\curl(Q_q)$ and $E_q$ which solve
\begin{align*}
    \partial_t (\nabla\Psi_q) + \grad\cdot\left( \nabla\Psi_q \otimes \grad^\perp\Psi_q \right) = \curl(Q_q) + E_q.
\end{align*}
At this stage, each function is supported in frequency in a ball of radius $\laq$ around the origin (ignoring the effect of the localizer $\Lq$ for the moment). The goal is to send $E_q$ to $0$ as $q\rightarrow \infty$, thus obtaining a solution to 3D QG in the limit. In order to minimize $E_q$, we essentially add a linear combination of building blocks $\nabla(\Lq\W)$ in the hopes of making
\begin{align}\label{cancellationintro}
\grad \cdot \left( \nabla(\Lq\W) \otimes \grad^\perp(\Lq\W) \right) - E_q 
\end{align}
vanish at low frequencies. In order to facilitate this cancellation, we first require an inverse divergence operator $\D$ satisfying 
$$ E_q = \grad \cdot \Mdot_q = \grad \cdot \left( \D (E_q) \right). $$
In order for \eqref{cancellationintro} to hold, $\D$ must output a matrix field $\Mdot_q$ which resembles a tensor product
$$\nabla(\Lq\W) \otimes \grad^\perp (\Lq\W).$$ 
Therefore, we must define the range of $\D$ to consist of matrices which have zeroes in the third row. In addition, considering that the divergence $\grad\cdot$ is in $x$ and $y$ only, it is natural for $\D$ to be a convolution operator in $x$ and $y$ only as well. After constructing such a $\D$ (see \cref{inversediv}), it is clear that the amount of regularity it gains will depend on only the first two components of the frequency modes of $E_q$. We will refer to these modes throughout the paper as the ``$x$ and $y$ frequency modes."  We emphasize that for vector fields $v=(v_1,v_2,v_3)$, the phrase ``$x$ and $y$
 frequency modes" refers to the first two components of the active frequencies (which are vectors with 3 components) in $v_1$, $v_2$, and $v_3$ rather than the active frequencies of $v_1$ and $v_2$.
 
One also notices that since $\nabla(\Lq\W)$ vanishes in a neighborhood of $z=0$ and $z=2\pi$, $\D(E_q)$ must vanish there as well. We therefore introduce an additional inductive assumption (see \eqref{inductivespatialsupport}); namely, that the spatial supports of $\nabla\Psi_q$, $\curl(Q_q)$, and $E_q$ are contained in the region where $\Lq=1$.  Since the inverse divergence is a convolution in $x$ and $y$ only, $\Mdot_q = \D(E_q)$ will only be supported in the region where $\Lq\equiv 1$ as well. Furthermore, the advection operator $\Dtq := \partial_t + \grad^\perp\Psi_q\cdot \grad$ applied to $\Lq$ satisfies
$$ \partial_t \Lq + \grad^\perp \Psi_q \cdot \grad \Lq =0 .$$
Thus, multiplication by $\Lq$ \textit{commutes} with the important operators in our scheme and does not interfere with the oscillatory term, making its implementation rather simple.

\subsubsection{Connection to 2D Euler}
Suppose that one were to construct a solution to 3D QG which did not depend on $z$.  While such a solution would then ignore all the important physical aspects of three-dimensional quasi-geostrophic dynamics, under this condition the equation becomes
$$ \partial_t \left( \grad \Psi \right) + \grad^\perp\Psi \cdot \grad \left( \grad \Psi ) \right) = \grad^\perp Q, $$
which after setting $u=\grad^\perp\Psi$ and $p=Q$ becomes 2D Euler. To construct solutions to 2D Euler using our scheme, we simply lift all restrictions on the spatial support, discard the localizer $\Lq$, and choose frequency modes with vanishing third component at each stage of the iteration so that $\partial_z \Psi \equiv 0$.  Thus, it is natural that our scheme for 3D QG should produce H\"{o}lder continuous solutions in classes $C^\zeta$ for $\zeta\in\left(0,\frac{1}{5}\right)$, as the Onsager conjecture for 2D Euler remains open in between $\frac{1}{5}$ and $\frac{1}{3}$. 

\subsubsection{An Onsager Conjecture for 3D QG} In \cite{novackweak}, it was shown that weak solutions to 3D QG conserve the energy $\|\nabla\Psi(t)\|_{L^2}^2$ when $\nabla\Psi$ belongs to the space
$ L^\infty_{t,z}\left( \Bdot^s_{3,\infty}\left( \mathbb{R}^2_{x,y}\right) \right) $ for $s>\frac{1}{3}$. The stratification of the velocity field allows for the lower regularity in the $z$ variable; essentially, one only needs to integrate by parts in $x$ and $y$ to show that the energy flux cannot contribute to the spontaneous production or dissipation of energy.  This leads us to conjecture the following dichotomy concerning the flexibility of weak solutions to 3D QG:
\vspace{.1cm}
\begin{changemargin}{2cm}{2cm}
\begin{center}\begin{spacing}{1.3}
\textit{For any $\zeta\in\left(0,\frac{1}{3}\right)$, there exists infinitely many weak solutions which do not conserve the energy $\| \nabla \Psi(t) \|_{L^2}$. In H\"{o}lder classes above $\frac{1}{3}$, the energy of a weak solution is constant in time.}
\end{spacing}\end{center}
\end{changemargin}
Therefore, \cref{maintheorem} addresses the Onsager conjecture for 3D QG only in the regime $\zeta\in(0,\frac{1}{5})$.    

\subsubsection{Relation of Our Result to Non-Uniqueness for 2D SQG}

The Onsager threshold for the inviscid SQG equation is conjectured to correspond to $\partial_z\Psi \in L^\infty$ and is not fully resolved yet (see \cite{bsv16} for a thorough discussion). As our solutions vanish at $z=0$ and $z=2\pi$, \cref{maintheorem} does not imply any results for 2D SQG. Nor does our result follow from the non-uniqueness of 2D SQG shown in \cite{bsv16}.  In 2D SQG, one has that $\Delta\Psi(t)\equiv 0$ for all time $t$. Physically, this represents an atmosphere which is at rest in the interior, and in which all the dynamics occur at the boundary. However, for 3D QG, one does not rule out the possibility of interior vorticity, allowing for the addition of high frequency oscillations not only at the boundary, but in the interior as well. The solutions we construct are not harmonic. Therefore, it is natural that they should be more regular than the dissipative solutions to 2D SQG.

\subsection{Background and Previous Results}\label{background}
Non-uniqueness of weak solutions to the Euler equations has been known for some time, with proofs given by Scheffer \cite{scheffer} and Shnirelman \cite{shnirelman1}. The modern convex integration techniques were developed by De Lellis and Sz\'{e}kelyhidi in \cite{delellisszek1}, \cite{dls13}, and \cite{dls14}. After a number of results investigating the flexibility of solutions and obtaining partial progress towards Onsager's conjecture for the 3D Euler equations (cf. \cite{Choffrut2013}, \cite{cs14}, \cite{buc15}, \cite{bdlisj15}, \cite{bdls16}), \cite{dsj17}, \cite{dls14}, \cite{ise13a}, \cite{ise13b}), a proof of the full conjecture was given by Isett \cite{isettonsager}. In a subsequent work, Buckmaster, De Lellis, Sz\'{e}kelyhidi, and Vicol \cite{bdlsv18} treat the case of dissipative solutions in the full Onsager regime. In \cite{isett17}, Isett constructed H\"{o}lder continuous solutions obeying the local energy inequality. In \cite{iv15}, Isett and Vicol demonstrate non-uniqueness of H\"{o}lder continuous weak solutions to active scalar equations with velocity determined by a Fourier multiplier which is not odd. Non-uniqueness for 2D SQG (an example of an active scalar equation with an odd multplier) was shown by Buckmaster, Shkoller, and Vicol \cite{bsv16}. In addition, non-uniqueness of 3D Navier-Stokes has been demonstrated by Buckmaster and Vicol \cite{bv172}, and Buckmaster, Colombo, and Vicol \cite{BCV2018}.  Stationary solutions to the 4D Navier-Stokes equations have been constructed by Luo \cite{luo18}, and the 3D Navier-Stokes equations by Cheskidov and Luo \cite{cl19}.  Examples of other settings in which convex integration has also been applied include the ideal MHD equations (by Beekie, Buckmaster, and Vicol \cite{bbv19}) or Hall-MHD equations (Dai \cite{dai}).  For a more thorough summary of these and other results, we refer the reader to the survey papers of Buckmaster and Vicol \cite{bv19} and De Lellis and Sz\'{e}kelyhidi \cite{dls19}.


Compactly supported H\'{o}lder continuous solutions to the Euler equations on $\mathbb{R}\times\mathbb{R}^3$ have been constructed by Isett and Oh in \cite{isettoh2016}.  One of the main difficulties of constructing such solutions lies in finding an anti-divergence operator which yields a smooth, symmetric tensor with compact support.  Isett and Oh provide a formula and derivation for such an operator, which in fact coincides with a formula introducted by Bogovskii \cite{bogovskii}. Application of this operator requires that special attention be paid to the angular momentum of the objects involved in the construction, as any weak solution to the Euler equations on $\mathbb{R}\times\mathbb{R}^n$ conserves linear and angular momentum (cf. Proposition 3.1) \cite{isettoh2016}.  The authors thus build a framework to ensure conservation of angular momentum throughout the scheme, in turn ensuring a good control over the spatial support of the stress at each stage of the iteration.  In constrast, the difficulties with spatial support in the current paper lie not in solving the divergence equation, but in finding stationary solutions to the system which in addition are capable of cancelling the stress.  

There are by now a number of significant results for both inviscid and viscous quasi-geostrophic flows on bounded, unbounded, or periodic domains.  Derivations of the three dimensional system in the upper half space have been offered by Bourgeois and Beale in the inviscid case \cite{bb} and Desjardins and Grenier in the viscous case \cite{dg}. Existence of global weak solutions in the inviscid case was first shown by Puel and Vasseur \cite{pv} for initial data belonging to Hilbert spaces. In \cite{novackweak}, several different notions of weak solutions were considered and shown to be equivalent under appropriate assumptions, with an existence proof being offered in the most general setting. In \cite{boundeddomains}, a formal derivation of the three dimensional system on a bounded cylindrical domain with appropriate lateral boundary conditions was given, with global weak solutions shown to exist satisfying those boundary conditions. Local-in-time classical solutions to this model were shown to exist in \cite{novack2019classical}. Global existence of a unique classical solution in the viscous case with spatial domain given by $\mathbb{R}_+^3$ was shown in \cite{novackvasseur}.  

The closely related 2D surface quasi-geostrophic equation (SQG) can be considered as a special case of the 3D QG system when $\Delta\Psi(t)\equiv 0$ for all times $t$. In this case, the dynamics is described by the active scalar equation for the unknown function $\theta:=-\partial_z\Psi|_{z=0}$ with velocity given by $u=\Riesz^\perp\theta$, ($\Riesz$ being the perpendicular vector of two dimensional Riesz transforms):
$$ \partial_t \theta + u \cdot \grad \theta + (-\overline{\Delta})^\alpha \theta = 0.$$
The physical cases correspond to $\alpha=0$ and $\alpha=\frac{1}{2}$. Study of this system, particularly the inviscid version, is extensive due to similarities with the 3D Euler equations and was initiated by Constantin, Majda, and Tabak \cite{cmt}. Global weak solutions in the inviscid case has been shown by Resnick \cite{Resnick} and Marchand \cite{Marchand}. Global existence of smooth solutions has been shown by a number of different methods by Kiselev, Nazarov, and Volberg \cite{knv}, Caffarelli and Vasseur \cite{cv}, Constantin and Vicol \cite{cvicol}, Constantin, Vicol, and Tarfulea \cite{cvicoltarfulea}, and Kiselev and Nazarov \cite{kn}. On bounded domains, a version of the equation defined using the spectral Riesz transform has been considered by Constantin and Ignatova in \cite{ci}, \cite{ci2}, Constantin and Nguyen in \cite{cn}, \cite{cn2}, Nguyen \cite{nguyen17}, and Constantin, Ignatova, and Nguyen \cite{cin18}. Non-uniqueness for both inviscid and viscous SQG on the spatial domain $\mathbb{T}^2$ was shown by Buckmaster, Shkoller, and Vicol \cite{bsv16}.

\section{Preliminaries}\label{Preliminaries}

We begin with definitions and some facts about H\"{o}lder spaces. At various points in these preliminaries, we use the notation $\lesssim$ as a shorthand for inequalities with implicit constants which depend on only on fixed quantities such as dimension and the notation $\approx$ to denote boundedness above and below up to fixed constants.  We emphasize that any usage of the symbol $\lesssim$ or $\approx$ in Sections 3-6 will denote dependence on implicit constants \textit{which are independent of $q$}, the parameter corresponding to the stage of the inductive convex integration procedure.

\begin{definition}[{\textbf{H\"{o}lder Spaces}}]\label{holder spaces}
Let $\alpha\in(0,1)$ and $k$ a non-negative integer and $f:\mathbb{R}\times\mathbb{T}^n\rightarrow\mathbb{R}$ a function of time and space with mean value zero on $\mathbb{T}^n$ for each fixed time.
\begin{enumerate}
    \item The integer spatial H\"{o}lder norms are defined by $$\| f \|_{C^k} = \sup_{t,x} \left|\nabla^k_x f(t,x)\right|.$$
    \item The non-integer spatial H\"{o}lder norms are defined by
    $$\| f \|_{C^{k,\alpha}} = \sup_{t,x,y} \frac{\left| \nabla^k_x f(t,x) - \nabla^k_x f(t,y) \right|}{|x-y|^\alpha} + \| f \|_{C^k}.$$
    \item The following interpolation inequality holds for $0\leq r \leq 1$.
    $$ \left\| f \right\|_{C^{r\alpha}} \leq C(\alpha) \left\| f \right\|_{C^0}^{1-r} \left\| f \right\|_{C^\alpha}^r. $$
\end{enumerate}
\end{definition}

We now define the convolution and projection operators we shall make use of and record some inequalities.  We divide them into two categories: kernels that depend on $x$, $y$, and $z$ and therefore act on functions whose domain is $\mathbb{T}^3$, and kernels that depend only on $x$ and $y$ and therefore act on functions defined on $\mathbb{T}^2$. At various points throughout the discussion, we will freely substitute definitions and proofs for operators defined on $\mathbb{R}^n$ rather than $\mathbb{T}^n$.  Standard transference arguments then provide analogous results for the periodic operators. In addition, all periodic functions are assumed to have mean zero. To simplify notation, we shall write sums over $\mathbb{Z}^3\setminus\{0\}$ as simply being over $\mathbb{Z}^3$, and analogously for $\mathbb{Z}^2$.  We shall require a Bernstein inequality.

\begin{lemma}[{\textbf{Bernstein Inequality}}]\label{bernstein}
Let $f:\mathbb{R}^n\rightarrow\mathbb{R}$ be a smooth function whose Fourier transform $\hat{f}$ vanishes in a neighborhood of the origin.  If $\hat{K}$ is a Fourier multiplier which is smooth away from the origin and homogeneous of degree $s$ and one of the following holds for $\lambda>0$,
\begin{enumerate}
    \item  $\supp \hat{f} \subset \{ | \xi | \leq \lambda \}$ and $s> 0$
    \item  $\supp \hat{f} \subset \{ | \xi | \geq \lambda \}$ and $s< 0$
    \item $\supp \hat{f} \subset \{ | \xi | \approx \lambda \}$ and $s\in \mathbb{R}$, 
\end{enumerate}
then 
    $$ \left\| \left(\hat{K}\hat{f}\right)^{\vee} \right\|_{C^0} \lesssim \lambda^s \| f \|_{C^0} .$$
\end{lemma}

\begin{definition}[{\textbf{$\torus$ Operators}}]\label{t3operators}
Let $f:\mathbb{T}^3\rightarrow \mathbb{R}$, $g: \mathbb{T}^3\rightarrow \mathbb{R}^3$ be smooth, mean-zero functions.
\begin{enumerate}
    \item The vector of $\torus$-Riesz transforms, denoted $\Riesz^3$, acts on Fourier series via
    $$ \Riesz^3 \left( f \right) = \sum_{\mathbb{Z}^3 } \frac{ik}{|k|}\hat{f}(k) e^{ik\cdot x} $$ and satisfies 
    $$ \left\| \Riesz^3(f) \right\|_{C^\alpha} \leq C(\alpha) \| f \|_{C^\alpha} $$
    for non-integer $\alpha>0$ or for $f$ with frequency support in an annulus. If $k$ is an integer, then
    $$ \left\| \Riesz^3 (f) \right\|_{C^k} \leq C(k,\alpha) \left\| f \right \|_{C^{k+\alpha}}. $$
    \item The projector onto gradients $\Pgrad$ is defined by
    $$ \Pgrad \left( g \right) := -\left( \Riesz^3 \otimes \Riesz^3 \right) (g)$$
    and satisfies the same estimates as $\Riesz^3$. 
    \item The projector onto curls $\Pcurl$ is defined by
    $$ \Pcurl (g) = \left( \Id - \Pgrad \right)(g) = \left( \curl \circ \left( - \Delta \right)^{-1} \circ \curl \right)(g) $$
    and satisfies the same estimates as $\Riesz^3$ and $\Pgrad$.
    \item Let $\lambda \in \mathbb{N}$ and $k \in \mathbb{S}^2 \cap \mathbb{Q}$ such that $\lambda k \in \mathbb{Z}^3$. The projector $\mathbb{P}_{\lambda, k}^\nabla$ is defined by $$ \mathbb{P}_{\lambda, k}^\nabla (g) = e^{i\lambda k \cdot x} \frac{ik}{|k|}\otimes\frac{ik}{|k|} \hat{g}(k)$$
    and satisfies 
    $$ \left\| \mathbb{P}_{\lambda, k}^\nabla (g) \right\|_{C^0} \leq \| g \|_{C^0}, \qquad \left\|\mathbb{P}_{\lambda, k}^\nabla (g)\right\|_{C^\alpha} \leq C(\alpha) \| g \|_{C^0} \lambda^\alpha. $$
    For $\lambda=\laqone$, we will denote this operator by 
    $\Pgradk$.
    \end{enumerate}
\end{definition}

We note that by our convention that all periodic functions have mean zero, we can identify $\Pgrad(g)$  with the gradient $\nabla \mathcal{G}$ of a uniquely defined mean-zero function $\mathcal{G}$. Similar properties hold for the operators $\Pcurl$ and $\Pgradk$.

\begin{definition}[{\textbf{$\mathbb{T}^2$ Operators}}]\label{t2operators}
Let $f:\mathbb{T}^2\rightarrow \mathbb{R}$, $g: \mathbb{T}^2\rightarrow \mathbb{R}^2$ be smooth, mean-zero functions.
\begin{enumerate}
    \item The vector of $\mathbb{T}^2$-Riesz transforms, denoted $\Riesz^2$, acts on Fourier series via
    $$ \Riesz^2 \left( f \right) = \sum_{\mathbb{Z}^2} \frac{ik}{|k|}\hat{f}(k) e^{ik\cdot x} $$ and satisfies 
    $$ \left\| \Riesz^2(f) \right\|_{C^\alpha} \lesssim \| f \|_{C^\alpha} $$
    for non-integer $\alpha>0$ or for $f$ with frequency support in an annulus. If $k$ is an integer, then
    $$ \left\| \Riesz^2 (f) \right\|_{C^k} \lesssim \left\| f \right \|_{C^{k+\alpha}} $$
    where the implicit constant depends on $\alpha>0$.
    \item The projector onto gradients $\Pgradbar$ is defined by
    $$ \Pgradbar \left( g \right) = -\left( \Riesz^2 \otimes \Riesz^2 \right) (g)$$
    and satisfies the same estimates as $\Riesz^2$. When $g=\left(g_1,g_2,g_3\right):\torus\rightarrow\mathbb{R}^3$, $\Pgradbar(g)$ projects on the first two components and is the identity on the third component.
    \item The projector onto perpendicular gradients $\Pgradbarperp$ is defined by
    $$ \Pgradbarperp(g) = \left( \Id - \Pgradbar \right)(g) = \left( \grad^\perp \circ \left( - \lap \right)^{-1} \circ (\grad^\perp \cdot) \right)(g) $$
    and satisfies the same estimates as $\Riesz^2$ and $\Pgradbar$.  When $g=\left(g_1,g_2,g_3\right):\torus\rightarrow\mathbb{R}^3$, $\Pgradbarperp(g)$ projects on the first two components and is zero in the third component.
    \item The inverse of $\grad^\perp$, denoted $\left( \grad^\perp \right)^{-1}$, is defined by
    $$  \left( \grad^\perp \right)^{-1}(g) = \left(-\lap\right)^{-1} \circ \left( \grad^\perp \cdot \right) (g). $$
    If the frequency support of $g$ is contained in an annulus of radius $\lambda$, then
    $$ \left\| \left(\grad^\perp\right)^{-1} g \right\|_{C^0} \lesssim \frac{1}{\lambda} \| g \|_{C^0}. $$
    \item Let $\lambda>0$ and define $\overline{\mathbb{P}}_{\approx \lambda}$ by 
    $$ \overline{\mathbb{P}}_{\approx \lambda}\left( f \right) = \sum_{\frac{\lambda}{2} \leq |k| \leq 2\lambda} \hat{f}(k)e^{ik\cdot x}.$$
    Define $\overline{\mathbb{P}}_{\leq \lambda}$ by 
    $$ \overline{\mathbb{P}}_{\leq \lambda}\left( f \right) = \sum_{|k| \leq 2\lambda} \hat{f}(k)e^{ik\cdot x}$$
    and $\overline{\mathbb{P}}_{\geq \lambda}$ similarly. Each operator is bounded from $C^\alpha$ to $C^\alpha$ for any $\alpha\geq 0$.
\end{enumerate}
\end{definition}

We shall frequently apply the $\mathbb{T}^2$ operators to functions $f:\mathbb{T}^3\rightarrow \mathbb{R}$. If $K$ is a $\mathbb{T}^2$ convolution operator, then by definition
$$ K(f)(x,y,z) = \int_{\mathbb{T}^2} K\left( x-s, y-t\right) f(s,t,z) \,ds\,dt. $$ The following lemma will be applied repeatedly throughout the paper.
\begin{lemma}\label{flatfrequencylocalizers}
Let $f:\mathbb{T}^3\rightarrow \mathbb{R}$ be a smooth function. Denote $k\in\mathbb{Z}^3$ by $\left( \bark, k_3 \right)$ and let $\lambda>0$.  Then $\overline{\mathbb{P}}_{\approx \lambda}(f)$ is supported in frequency in the cylinder
$$ \mathcal{C}_\lambda = \left\{k \in \mathbb{Z}^3 : | \bark | = |(k_1,k_2)| \approx \lambda, k_3 \in \mathbb{Z} \right\}. $$
If $\supp \hat{f} \subset \mathcal{C}_\lambda$, then $\overline{\mathbb{P}}_{\approx \lambda}(f) = f$.  Furthermore, analogous statements hold for $\overline{\mathbb{P}}_{\leq \lambda}$ and $\overline{\mathbb{P}}_{\geq \lambda}$ by replacing $\approx$ with $\leq$ and $\geq$, respectively.
\end{lemma}

\begin{proof}
Fix $z\in [0,2\pi]$. For $(x,y,z)\in\torus$, we denote $(x,y,0)$ by $\bar{x}$.  Then 
$$ \overline{\mathbb{P}}_{\approx \lambda}(f)(x,y,z) = \sum_{\kbar \approx \lambda} \hat{f}(\kbar,z)e^{i\kbar\cdot \bar{x}} $$
where
$$ \hat{f}\left( \bark, z \right) = \frac{1}{(2\pi)^2} \int_{\mathbb{T}^2} f(x,y,z) e^{i\kbar \cdot \bar{x}} \,dx \,dy. $$
Letting $z$ vary, we have that $\hat{f}\left(\bark, z\right)$ is a smooth function of $z$ and can therefore be written as
$$ \hat{f}\left( \kbar, z \right) = \sum_{k_3 \in \mathbb{Z}} \hat{a}\left( \kbar, k_3 \right) e^{i k_3 \cdot z}. $$
Combining the series, we have
$$ \overline{\mathbb{P}}_{\approx \lambda}(f)(x,y,z) = \sum_{|\kbar| \approx \lambda} \sum_{k_3 \in \mathbb{Z}} \hat{a}(\kbar, k_3) e^{i k_3 \cdot z} e^{i \kbar \cdot \bar{x}} =: \sum_{k \in \mathcal{C}_\lambda} \hat{a}(k) e^{ik \cdot x}. $$
By the uniqueness of $\torus$ Fourier coefficients, if $\supp \hat{f} \subset \mathcal{C}_\lambda$, then $\hat{a}(k) = \hat{f}(k)$, and $\overline{\mathbb{P}}_{\approx \lambda}(f) = f$.
\end{proof}

We now define our inverse divergence operator $\D$, which will be a convolution kernel in $x$ and $y$ only. We define the operator on functions whose spatial domain is $\mathbb{T}^3$, however, since it will only be applied to such functions in this paper.  The amount of regularity gained by $\D$ depends only on the $x$ and $y$ frequency modes. 

\begin{prop}[{\textbf{Inverse Divergence of $\grad$}}]
Let $\grad f:\torus \rightarrow \mathbb{R}^3$ have zero mean on each slice $\{z\}\times\mathbb{T}^2$ for $z\in[0,2\pi]$. Define $\E(\grad f)$ by 
\[
   \E(\grad f)=
  \left[ {\begin{array}{cc}
   \partial_{22} \inverselap f - \partial_{11} \inverselap f & -2\partial_{12} \inverselap f \\
   -2\partial_{12} \inverselap f & \partial_{11} \inverselap f - \partial_{22} \inverselap f \\
  \end{array} } \right]
\]
Then $\grad\cdot\E(\grad f) = \grad f$, and $\E(\grad f)$ is symmetric and traceless.  If $\supp \hat{f} \subset \left\{ |\kbar| \geq \lambda \right\}$, then
$$ \left\| \mathcal{E}(\grad f) \right\|_{C^0} \lesssim \frac{1}{\lambda} \| \grad f \|_{C^0}. $$
\end{prop}

\begin{proof}
The equality of $\grad \cdot \mathcal{E}( \grad f )$ and $\grad f$ proceeds by direct computation.  The estimate on the $C^0$ norm follows from first applying \cref{flatfrequencylocalizers} to see that for each $z\in[0,2\pi]$, $\grad f(\cdot,\cdot,z)$ has frequency support outside the set $\{(k_1,k_2) \in \mathbb{Z}^2 : \sqrt{k_1^2 + k_2^2} \geq \lambda\}$.  Then using \cref{bernstein} and the fact that the multiplier is homogeneous of degree $-1$ gives the claim. Notice that $\mathcal{E}$ is identical to the inverse divergence operator defined in \cite{bsv16} after switching the rows and changing the sign of the new second row.
\end{proof}

\begin{prop}[{\textbf{Inverse Divergence of Scalar Functions}}]
Let $ g:\torus\rightarrow\mathbb{R}$ have zero mean for each slice $\{z\}\times\mathbb{T}^2$ for $z\in[0,2\pi]$.  Define $\mathcal{I}(g)$ by 
$$ \mathcal{I}( g) = -\inverselap \grad  g .$$
Then $\grad \cdot \mathcal{I}(g) = g$.  If $\supp \hat{g} \subset \left\{ |\kbar| \geq \lambda \right\}$, then
$$ \left\| \mathcal{I}(g) \right\|_{C^0} \lesssim \frac{1}{\lambda} \| g \|_{C^0}. $$
\end{prop}

\begin{proof}
As before, the equality proceeds by direct computation and the estimate is a corollary of \cref{flatfrequencylocalizers}, \cref{bernstein} and the homogeneity of the symbol.
\end{proof}

The inverse divergence we use will be applied to vector fields for which the first two components are the gradient of a scalar-valued function, while the third component is a (different) scalar-valued function.

\begin{prop}[{\textbf{Inverse Divergence of $\X := (\partial_x f, \partial_y f, g)$}}]\label{inversediv}
Let $\X=(\grad f, g):\torus\rightarrow \mathbb{R}^3$ have zero mean for each slice $\{z\}\times\mathbb{T}^2$ for $z\in[0,2\pi]$.  Define 
$\D(\X)$ to be the $3\times 3$ matrix 
\[
   \D(\X)= \inverselap
  \left[ {\begin{array}{ccc}
   \left(\partial_{22}  - \partial_{11}\right)\left(f\right) & -2\partial_{12} \left(f\right) & 0 \\
   -2\partial_{12}\left(f\right) &  \left(\partial_{11}-\partial_{22}\right)\left(f\right)  & 0 \\
   -\partial_{1} g   & -\partial_{2} g & 0 \\
   \end{array} } \right]
\]
Then $\grad\cdot \D(\X) = \X$.   If $\supp \hat{\X} \subset \left\{ |\kbar| \geq \lambda \right\}\times\mathbb{Z}$, then
$$ \left\| \mathcal{D}(\X) \right\|_{C^0} \lesssim \frac{1}{\lambda} \| \X \|_{C^0}. $$
\end{prop}
\begin{proof}
The equality of $\grad \cdot \D \X$ and $\X$ proceeds by direct computation.  The estimate follows from \cref{flatfrequencylocalizers} and \cref{bernstein} as previously.
\end{proof}

The following lemma is the analogue of the so-called geometric lemma from \cite{bsv16} and describes the mechanism by which we can cancel out errors with the addition of high-frequency waves. The choice of two disjoint sets of frequencies $\Omega_1$ and $\Omega_2$ ensures that high-frequency waves $\V_1$ and $\V_2$ defined on time intervals which overlap do not produce unwanted low-frequency interactions.  Choosing $\V_1$ to oscillate at frequencies belonging to $\Omega_1$ and $\V_2$ to oscillate at frequencies belonging to $\Omega_2$ means that even if $\V_1\otimes\V_2 \neq 0$, it is at least high-frequency and will enjoy a strong gain from the application of the anti-divergence operator.

\begin{lemma}[{\textbf{Choosing Frequency Modes}}]\label{frequencymodes}
Define
 \[
\mathcal{M}= \left\{  \left[ {\begin{array}{ccc}
m_1 & m_2 & 0\\
m_3 & -m_1 & 0\\
m_4 & m_5 & 0\\
\end{array}} \right] : m_i \in \mathbb{R} \right\}
\] 
Then there exist matrices $M_1, M_2 \in\mathcal{M}$, $\epsilon >0$, disjoint finite subsets $\Omega_j \in \mathbb{Q}^3 \cap \mathbb{S}^2$ for $j=1,2$, and smooth positive functions defined in a neighborhood of $M_j$ and indexed by $k\in\Omega_j$ which we call $c_{j,k} \in C^\infty\left(B_{\epsilon,\mathcal{M}}(M_j)\right)$ such that
\begin{enumerate}
\item Both of the sets $\Omega_j$ are at positive distance from the $z$-axis
\item $\Omega_j = -\Omega_{j}$ and $c_{j,k} = c_{j,-k}$
\item $13\Omega_j \in \mathbb{Z}^3$ for $j=1,2$
\item For $j=1,2$ and $\forall M \in B_\epsilon(M_j)$, we have
$$ M = \frac{1}{2}\sum_{k\in\Omega_j} {(c_{j,k}(M))^2 k\otimes \bark^\perp}.$$
\item Furthermore, if $M=M_j+N$ where $N\in \mathcal{M}$ satisfies $N^{12}=N^{21}$ (i.e., the top left block of $N$ is symmetric in addition to being traceless), then 
\begin{align}\label{aklidentity}
\sum_{k\in \Omega_j} \left(c_{j,k}(M)\right)^2 = 1. 
\end{align}
\end{enumerate}
\end{lemma}

\begin{proof}
We begin by constructing $\Omega_1^+$, where $\Omega_1^-$ will be defined as $-\Omega_1^+$ and $\Omega_1 = \Omega_1^+ \cup \Omega_1^-$. We choose the following vectors (inspired by the fact that $(5,12,13)$ and $(3,4,5)$ are Pythagorean triples):
$$k_1 = \frac{1}{13}(5,0,12), \quad k_2 = \frac{1}{13}(3,4,-12), \quad k_3 = \frac{1}{13}(3,-4,12),$$ 
$$\quad k_4 = \frac{1}{13}(0,5,12)\quad k_5 = \frac{1}{13}(3,4,12).  $$
Then it is clear that (1) and (3) hold for $\Omega_1^+$. Constructing the corresponding matrices $k_i \otimes \bark^\perp_i$, denoted $m_{k_i}$, we have
\[
m_{k_1}=  \frac{1}{169} \left[ {\begin{array}{ccc}
0 & 25 & 0 \\
0 & 0 & 0\\
0 & 60 & 0\\
\end{array}} \right] , 
\quad
m_{k_2}=  \frac{1}{169} \left[ {\begin{array}{ccc}
 -12 & 9 & 0 \\
-16 & 12 & 0\\
48 & -36 & 0\\
\end{array}} \right]  , \quad
m_{k_3}= \frac{1}{169}  \left[ {\begin{array}{ccc}
12 & 9 & 0 \\
-16 & -12 & 0\\
48 & 36 & 0\\
\end{array}} \right]  
\] 
\[
m_{k_4}= \frac{1}{169}  \left[ {\begin{array}{ccc}
0 & 0 & 0 \\
-25 & 0 & 0\\
 -60 & 0 & 0\\
\end{array}} \right]  , \quad
m_{k_5}= \frac{1}{169}  \left[ {\begin{array}{ccc}
-12 & 9 & 0 \\
-16 & 12 & 0\\
 -48 & 36 & 0\\
\end{array}} \right]  
\] 
Furthermore, one can check that the set $\{m_{k_i}\}$ is a linearly independent set within $\mathcal{M}$. After identifying $\mathcal{M}$ with $\mathbb{R}^5$, define the function $f_1:\mathbb{R}^5\rightarrow\mathbb{R}^5$ by 
$$ f_1(x,y,z,s,t) = xm_{k_1} + y m_{k_2} + zm_{k_3} + sm_{k_4} + tm_{k_5}. $$
Then $f_1\in C^\infty$ and $Df_1|_{\left(\frac{1}{10},\frac{1}{10},\frac{1}{10},\frac{1}{10},\frac{1}{10}\right)}$ is invertible.  Define $M_1:= f_1\left(\frac{1}{10},\frac{1}{10},\frac{1}{10},\frac{1}{10},\frac{1}{10}\right)$. Applying the inverse function theorem, we obtain $\epsilon_1$ and coefficient functions $c_{1,k}$.  Then adding the set of vectors $\Omega_1^-= \cup_i (-k_i)$, we have $\Omega_1$ such that (1)-(4) are satisfied. To show that (5) is satisfied, we note that given $M$ of such a form, then 
$$ M^{12} - M^{21} = M_1^{12}-M_1^{21}, $$
and that 
$$m_{k_i}^{12}-m_{k_i}^{21}=\frac{25}{169}$$
for each $k_i\in\Omega_1$. Therefore,
\begin{align*}
    \sum_{k\in\Omega_1} \left(c_{1,k}(M)\right)^2 &= \frac{169}{25}\sum_{k\in\Omega_1} \left(c_{1,k}(M)\right)^2 \frac{25}{169}\\
    &= \frac{169}{25}\sum_{k\in\Omega_1} \left(c_{1,k}(M)\right)^2 \left( m_k^{12}-m_k^{21} \right)\\
    &= \frac{169}{25}\sum_{k\in\Omega_1} \left(c_{1,k}(M_1)\right)^2 \left( m_k^{12}-m_k^{21} \right)\\
    &=  2 \sum_{k\in\Omega_1^+} \left(c_{1,k}(M_1)\right)^2\\
    &= 2\cdot5\cdot\frac{1}{10}\\
    &= 1,
\end{align*}
and thus (1)-(5) are satisfied for $\Omega_1$. To construct $\Omega_2$, replace each vector $k=(k_1,k_2,k_3)$ with $k'=(-k_2, k_1, k_3)$. Repeating the previous steps and taking the minimum of $\epsilon_1$ and $\epsilon_2$ finishes the proof.
\end{proof}
With the choice of frequency modes in hand, we can build the following approximately stationary solutions.

\begin{lemma}[{\textbf{Stationary Solutions}}]\label{stationarysolutions}
For a finite family of vectors $\Omega\in \mathbb{S}^2$ where $\Omega = -\Omega$, $\lambda \in \mathbb{N}$ such that $\lambda\Omega\in\mathbb{Z}^3$, and constants $c_k \in \mathbb{C}$ indexed by $k\in\Omega$ such that $c_k=\overline{c_{-k}}$, define 
$$ \mathbb{V}(x) := \sum_{k\in\Omega} \frac{1}{\lambda} c_k e^{i\lambda k\cdot x}. $$
Then $\mathbb{V}$ is real-valued and there exists $Q$ such that $\grad \cdot \left( \nabla \mathbb{V} \otimes \grad^\perp\mathbb{V} \right)=\curl(Q)$, with $Q$ obeying the bounds
$$ \left\| Q \right\|_{C^0} \lesssim \left\| (\nabla \V)^2 \right\|_{C^0}, \qquad \left\| \curl(Q) \right\|_{C^0} \lesssim \lambda \left\| (\nabla\V)^2 \right\|_{C^0} . $$
The mean of $\nabla\V\otimes\grad^\perp\V$ is given by
$$\frac{1}{2} \nabla \mathbb{V} \otimes \grad^\perp\mathbb{V} = \sum_{k\in \Omega} |c_k|^2 \left(k \otimes \kbar^\perp\right).$$
Furthermore, if $L(z)$ is a smooth function depending only on $z$, then
\begin{align*}
    \grad \cdot \left( \nabla(L\V)\otimes \grad^\perp(L\V) \right) &= L^2 \grad \cdot \left( \nabla \V \otimes \grad^\perp \V \right)\\
    &= \curl \left(L^2 Q\right) - \left( Q^2 \partial_z(L^2), - Q^1 \partial_z(L^2), 0 \right)^t.
\end{align*}
\end{lemma}
\begin{proof}
First note that $\mathbb{V}$ is real-valued by the assumptions on $c_k$. Then, we have that
$$ \Delta (c_k e^{i\lambda k\cdot x}) = -\lambda^2 c_k e^{i\lambda k\cdot x}; $$
that is, $\mathbb{V}$ is an eigenfunction of $\Delta$ with eigenvalue $-\lambda^2$. In order to show that $$ \grad \cdot \left( \nabla \mathbb{V} \otimes \grad^\perp\mathbb{V} \right)$$ is the curl of a vector field, it suffices to show that the divergence is zero. Calculating the divergence, we have
$$ \nabla \cdot \left (\grad \cdot \left( \nabla \mathbb{V} \otimes \grad^\perp\mathbb{V} \right)\right) = \grad\nabla\V : \grad^\perp\nabla\V +  \grad^\perp \mathbb{V} \cdot \grad (\Delta \mathbb{V}) = -\lambda^2 \grad^\perp \mathbb{V} \cdot \grad \mathbb{V} = 0.$$
After writing out $\nabla\V$ and $\grad^\perp\V$ in terms of Fourier series with modes $k$ and $k'$, respectively, we can restate this fact in the form of the following algebraic identity which will be crucial later in the paper.  
\begin{align}\label{algebraicidentity}
    \sum_{k,k' \in \Omega} c_k c_{k'} e^{i\lambda(k+k')\cdot x} (i\kbarperpprime \cdot i k) (ik \cdot i(k+k')) \lambda^2 = 0 \qquad \forall x \in \torus.
\end{align}
The bounds on $Q$ come from noticing that $Q$ solves the elliptic equation
$$ Q = (-\Delta)^{-1} \curl \left( \grad \cdot \left( \nabla \V \otimes \grad^\perp \V \right) \right) $$
and using the frequency support of $\V$ in conjunction with \cref{bernstein} to conclude that the singular integral operator $(-\Delta)^{-1} \circ \curl \circ \grad \cdot$ is bounded on $C^0$. By direct calculation, the low frequency portion of $\nabla\V\otimes\grad^\perp\V$  is given as stated.

Given a smooth function $L(z)$, it is clear that
$$\grad \cdot \left( (L \grad V) \otimes \grad^\perp (L\V) \right) = L^2 \grad \cdot \left( \grad \V \otimes \grad^\perp \V \right)$$ since $L$ depends only on $z$. We calculate the third component by writing
\begin{align*}
    \grad \cdot \left( \grad^\perp (L\V) \partial_z(L\V) \right) &= \grad \cdot \left( \grad^\perp (L\V) \partial_z L \V + \grad^\perp (L\V) L \partial_z \V \right)\\
    & = L\partial_z L \left( \grad^\perp \V \cdot \grad \V \right) + L^2 \grad^\perp\V \cdot \grad (\partial_z \V)\\
    &= L^2 \grad \cdot \left(  \partial_z \V \grad^\perp \V \right).
\end{align*}
Therefore, 
\begin{align*}
    \grad \cdot \left( \nabla(L\V)\otimes \grad^\perp(L\V) \right) &= L^2 \grad \cdot \left( \nabla \V \otimes \grad^\perp \V \right)\\
    &= L^2 \curl(Q) \\
    &= \curl (L^2 Q) - \left( Q^2 \partial_z(L^2), - Q^1 \partial_z(L^2), 0 \right)^t
\end{align*}
after recalling that $L$ depends on $z$ only.
\end{proof}

\section{Convex Integration Scheme}\label{Convexintegrationscheme}

\subsection{Inductive Assumptions}
We assume the existence of a triple $(\nabla\Psi_q, Q_q, \M_q)$ solving
\begin{align}\label{inductiveequation}
    \partial_t (\nabla\Psi_q) + \grad\cdot\left( \nabla\Psi_q \otimes \grad^\perp\Psi_q \right) = \curl(Q_q) + \grad \cdot \M_q.
\end{align}
The gradient of the stream function $\nabla\Psi_q$, the curl $Q_q$, and the matrix field $\M_q$ are assumed to be supported in frequency in the set
\begin{align}\label{inductivefrequencysupport}
 \{ k\in \mathbb{Z}^3 : |(k_1,k_2)| \leq \laq , k_3 \in \mathbb{Z}\}
\end{align}
The gradient of the stream function $\nabla\Psi_q$, the curl $Q_q$, and the matrix field $\M_q$ are assumed to be supported in space in the set
\begin{align}\label{inductivespatialsupport}
    \mathbb{T}^2 \times \left[\frac{1}{l_q}, 2\pi - \frac{1}{l_q}\right].
\end{align}
for a strictly positive number $l_q$. We assume that
\begin{align}\label{inductivevelocity}
    \left\| \nabla \Psi_q \right\|_{C^0} \lesssim 1, \qquad \left\| \nabla \Psi_q \right\|_{C^n} \leq  \dq^\frac{1}{2} \laq^n \quad \forall n\geq 1.
\end{align}
We assume that $\Mdot_q$ satisfies 
\begin{align}\label{inductiveerror}
    \left\| \M_q \right\|_{C^0} \leq \eta \dqone, \qquad \left\| \M_q \right\|_{C^1} \leq \dqone \laq.
\end{align}
In addition, we assume that the material derivative of $\M_q$ satisfies
\begin{align}\label{inductivetransport}
    \left\| \left(\partial_t + \grad^\perp\Psi_q \cdot \grad \right) \Mdot_q \right\|_{C^0} \leq \dqone \dq^\frac{1}{2} \laq.
\end{align}
We assume that $Q_q$ satisfies
\begin{align}\label{inductivecurl}
     \qquad \left\| Q \right\|_{C^0} \lesssim 1, \qquad \left\| \nabla Q \right\|_{C^0} \leq \dq \laq.
\end{align}
Concerning the prescribed energy profile $e(t)$, we assume that 
\begin{align}\label{inductiveenergyprofileone}
    0 \leq e(t) - \int_{\torus} |\nabla\Psi_q(t)|^2 \,dx \leq \dqone
\end{align}
and
\begin{align}\label{inductiveenergyprofiletwo}
    e(t) - \int_{\torus} |\nabla\Psi_q(t)|^2 \,dx \leq \frac{\dqone}{8} \Rightarrow \M_q(\cdot,t) \equiv 0.
\end{align}
The bulk of the paper consists of verifying that we can construct a triple $(\nabla\Psi_{q+1}, Q_{q+1}, \M_{q+1})$ satisfying \eqref{inductiveequation}-\eqref{inductiveenergyprofiletwo} with $q$ replaced with $q+1$ and parameters $\dqone<\dq$, $\laqone>\laq$, and $l_{q+1}<l_{q}$, where $\dq,l_q \rightarrow 0$ and $\laq \rightarrow 0$ as $q\rightarrow \infty$ at rates implying the desired level of H\"{o}lder regularity.

\subsection{Velocity Perturbation}
\subsubsection{A Spatial Localizer, Time Partition, Transport}
Define the cutoff function $\Lq$ to be a smooth function depending only on $z$ which satisfies
\begin{align}\label{cutofffunction}
    0 \leq \Lq(z) \leq 1&, \quad \Lq = 1 \quad \forall (x,y,z) \in \mathbb{T}^2\times[\frac{1}{l_{q+1}}, 2\pi - \frac{1}{l_{q+1}}], \\ 
    & \|\partial_z \Lq \|_{C^0} \lesssim {l_{q+1}}, \quad \supp \Lq \subset \mathbb{T}^2 \times [\frac{1}{l_{q+2}}, 2\pi - \frac{1}{l_{q+2}}].  \nonumber
\end{align}

Let $\mathcal{X} \in C_c^\infty\left( (-\frac{3}{4}, \frac{3}{4}) \right)$ be a smooth positive cutoff function such that
$$ \sum_{l \in \mathbb{Z}} \mathcal{X}^2 (x-l) =1 $$
for all $t\in\mathbb{R}$. Let the support of the energy profile $e(t)$ be contained in a ball of radius $R.$ For $\muqone$ a large parameter to be specified later and $l \in \mathbb{Z} \cap [-R\muqone,R\muqone]$, define (we neglect to indicate the dependence on $q$ for ease of notation)
$$ \Xl(t) := \mathcal{X}(\muqone t - l). $$
Define 
\begin{align}\label{definitionofrho}
    \rho(t):= \left( \int_{\torus} \Lq^2 \right)^{-1} \max\left(e(t) - \int_{\mathbb{T}^3} \|\nabla\Psi_q\|^2 - \frac{\dqtwo}{2}, 0  \right), \qquad \rho_l = \rho(\frac{l}{\muqone}).
\end{align}
By the assumption \eqref{inductiveenergyprofileone}, we have that
\begin{align}\label{inductiverhol}
   0 \leq \rho_l \leq \dqone.
\end{align}

Let $\phi_{q}(z)$ be a mollifier in $z$ which is compactly supported in a ball of radius $\ell = \laq^{-\frac{3}{4}}\laqone^{-\frac{1}{4}}$.  Define
$$\M_{q,\ell} = \Mdot_q \ast \phi_{q}$$
so that the spatial support of $\M_{q,\ell}$ is still contained in the region where $\Lq\equiv 1$ and
$$ \left\| \M_{q,\ell} \right\|_{C^0} \leq \eta \dqone, \qquad \left\| \M_{q,\ell} \right\|_{C^1} \leq \dqone \laq, \qquad \left\| \M_{q,\ell} \right\|_{C^n} \leq \dqone\laq \ell^{1-n} \quad \forall n \geq 2. $$
Let $\M_{q,l}$ be the unique solution to the transport equation
\[\begin{dcases}
        \partial_t \M_{q,l} + \grad^\perp \Psi_q \cdot \grad \M_{q,l} = 0 \\
        \M_{q,l}(x, \frac{l}{\muqone}) = \M_{q,\ell}(x, \frac{l}{\muqone}), \\
       \end{dcases}
\]
and set 
$$ M_{q,l} := \rho_l M_j - \M_{q,l} $$
where $M_j$ comes from \cref{frequencymodes}, and $j$ is chosen so that the parity of $l$ and $j$ matches. Next, let $\Phi_l:\mathbb{R}\times\mathbb{R}^3\rightarrow\mathbb{R}^3$ be the solution of 
\[\begin{dcases}
        \partial_t \Phi_{l} + \grad^\perp \Psi_q \cdot \grad \Phi_{l} = 0 \\
        \Phi_{l}\left(x, \frac{l}{\mu}\right) = x \\
       \end{dcases}
\]
so that $\Phi_l(\cdot,t)$ is a diffeomorphism of $\mathbb{T}^3$ onto itself, and for $(t,x)\in\mathbb{R}\times\mathbb{T}^3$ the map 
$$ (x,t) \rightarrow e^{i\lambda_{q+1}k\cdot\Phi_l(x,t)} $$
is well-defined.  

\subsubsection{The Perturbation}

Note that by \cref{inversediv} and \cref{frequencymodes}, $M_{q,l}=\rho_l M_j - \Mdot_{q,l}$ takes values in $\mathcal{M}$.  Let $k\in\Omega=\Omega_1\cup\Omega_2$ denote a chosen frequency mode.  Now 
define
$$ \mathcal{X}_l(t) := \mathcal{X}\left(\muqone(t-l)\right) $$
\[a_{kl}(x,t) :=
\begin{dcases}
   \sqrt[]{\rho_l} c_{j, k} \left( \frac{M_{q,l}(x,t)}{\rho_l} \right) \quad & \textnormal{if } \rho_l \neq 0 \\
  0 & \textnormal{if } \rho_l =0
  \end{dcases}\]
$$  w_{kl}(x,t) := a_{kl}(x,t)e^{i\lambda_{q+1}k\cdot\Phi_l(x,t)} ik.  $$
where $j=1$ and $k\in \Omega_1$ if $l$ is odd, and $j=2$ and $k\in \Omega_2$ if $l$ is even.  We must check that $\akl$ is well-defined when $\rho_l \neq 0$. Therefore we must check that given $\epsilon$ as in \cref{frequencymodes},
$$ \left| \frac{\M_{q,l}}{\rho_l} \right| < \epsilon. $$
Since $\Mdot_{q,l}$ satisfies a transport equation, it suffices to check that 
$$ \left\| \M_{q,\ell}\left( x, \frac{l}{\muqone} \right) \right\|_{C^0} \rho_l^{-1} \leq \epsilon. $$
By \eqref{inductiveenergyprofiletwo}, we have that
\begin{equation}\label{eitheror}
\rho_l \leq \frac{\dqone}{16} \Rightarrow \M_q \equiv 0.
\end{equation}
So we move to the case $\rho_l > \frac{\dqone}{16}$. Then 
$$ \frac{\Mdot_{q,l}}{\rho_l} \leq \frac{\eta \dqone}{\frac{\dqone}{16}} $$
which is less than $\epsilon$ as long as $\eta$ is small enough. $\nabla\W$ is now well-defined by (using the definition of $\Pgradk$ given in \cref{t3operators} and recalling that $\W$ is chosen to have mean zero)
$$ \nabla\W(x,t) := \sum_{l \textnormal{ odd}, k \in \Omega_1} \Pgradk \big{(}\mathcal{X}_l(t) w_{kl}(x,t)\big{)} \quad + \sum_{l \textnormal{ even}, k \in \Omega_2} \Pgradk\big{(}\mathcal{X}_l(t) w_{kl}(x,t)\big{)}. $$
Throughout the rest of the paper, we will simply write
$$  \nabla\W(x,t) = \sum_{l,k} \Pgradk \big{(}\mathcal{X}_l(t) w_{kl}(x,t)\big{)} $$
for the sake of simplicity. The perturbation is then defined by $\nabla(\W \Lq)$.

\subsection{Adding the Perturbation}
Define $\nabla\Psi_{q+1} = \nabla\Psi_q + \pert$. Using that $\nabla\Psi_q$ solves \begin{align*}
    \partial_t (\nabla\Psi_q) + \grad\cdot\left( \nabla\Psi_q \otimes \grad^\perp\Psi_q \right) = \curl(Q_q) + \grad \cdot \M_q,
\end{align*}
we have that $\nabla\Psi_{q+1}$ solves
\begin{align}
    \partial_t (\nabla\Psi_{q+1}) &+ \grad\cdot\left( \nabla\Psi_{q+1} \otimes \grad^\perp\Psi_{q+1} \right) = \curl(Q_q) \nonumber\\
    &+ \partial_t(\pert) + \grad^\perp\Psi_q \cdot \grad \pert \tag{Transport Error}\\
    &+ \pertperp \cdot \grad \nabla\Psi_q \tag{Nash Error}\\
    &+\grad \cdot \left( \pert \otimes \pertperp \right) + \grad \cdot \M_q \tag{Oscillation Error}\\
    &= \curl(Q_{q+1}) + \grad \cdot \left( \Mdot_{q+1} \right). \nonumber
\end{align}
The definition of the matrix field $\Mdot_{q+1}$ and the vector field $Q_{q+1}$ will be specified in the following sections.

\subsection{Choice of Parameters}

We define the parameters $\laq$, $\dq$, $\muqone$, and $l_q$ for all $q\in\mathbb{N}$ in terms of a real number $c>\frac{5}{2}$, a real number $b>1$, and a large integer $a\in 13\mathbb{Z}$.  The numbers $c$, $b$, and $a$ are chosen in that order after first fixing a desired H\"{o}lder regularity level $\zeta \in (0,\frac{1}{5})$. 
$$ \laq := a^{cb^q},\qquad  \dq := a^{-b^q}, \qquad \muqone := \dq^\frac{1}{4} \dqone^\frac{1}{4} \laq^\frac{1}{2} \laqone^\frac{1}{2}, \qquad l_q := \frac{1}{2^{q+1}} $$
In addition, $a$ will be chosen to be large enough to absorb universal constants coming from many of the steps of the argument. We also implement small parameters 
$$0< \alpha \ll \beta \ll 1$$
which are essentially used to control singular integral operators on $L^\infty$ and to quantify the super-exponential growth of the $\laq$'s.
With these choices, the following inequalities hold.

\begin{lemma}[\textbf{Parameter Inequalities}]\label{parameterinequalities}
Given $\laq$, $\dq$, $\muqone$, and $l_q$ as defined above, the following inequalities are true for satisfactory choices of $c$, $b$, $a$, $\beta$, and $\alpha$.
\begin{enumerate}
\item 
$\frac{\dq^\frac{1}{2}\laq}{\muqone} \leq \laqone^{-\beta}$
\item $\muqone \dqone^\frac{1}{2}\leq \dqtwo \laqone.$
\item $l_{q+1}\frac{\dqone}{\laqone}\leq \eta\dqtwo$
\item $\dqone\dq^\frac{1}{2}\laq \leq \dqtwo \dqone^\frac{1}{2}\laqone$
\item $\laq \leq \laqone^{1-\beta}$
\item $\laq^{1+\alpha} \leq \laqone^{1-
\alpha}$
\item  $\ell \lambda_q \leq 1$
\end{enumerate}
\end{lemma}

\begin{proof}
Writing out (1), we see that it is satisfied provided that
$$ \frac{a^{-\frac{1}{2}b^q}a^{cb^q}}{a^{-\frac{1}{4}b^q}a^{\frac{1}{2}cb^q}a^{-\frac{1}{4}b^{q+1}}a^{\frac{1}{2}cb^{q+1}}} \leq a^{-\beta cb^{q+1}}.$$
Taking the logarithm in $a$ of both sides and dividing by $b^q$ yields
\begin{align*}
&-\frac{1}{2}+c+\frac{1}{4}+\frac{1}{4}b - \frac{1}{2}c - \frac{1}{2}cb \leq -\beta cb\\
&\iff b\left( \frac{1}{4}-\left(\frac{1}{2}-\beta\right)c\right)+\frac{1}{2}c-\frac{1}{4}\leq 0,
\end{align*}
which is true if $\beta$ is small enough. The second inequality is true provided that
$$ a^{-\frac{1}{4}b^q}a^{\frac{1}{2}cb^q}a^{-\frac{1}{4}b^{q+1}}a^{\frac{1}{2}cb^{q+1}} a^{-\frac{1}{2}b^{q+1}} \leq a^{-b^{q+2}} a^{cb^{q+1}}. $$
Taking logarithms in $a$ and dividing by $b^q$ again gives
\begin{align*}
    &-\frac{1}{4}+\frac{1}{2}c-\frac{1}{4}b+\frac{1}{2}cb-\frac{1}{2}b \leq -b^2 + cb\\
    &\iff b^2 - b\left( \frac{3}{4}+\frac{1}{2}c \right) - \frac{1}{4}+\frac{1}{2}c \leq 0,
\end{align*}
which is true provided $b$ is close enough to $1$.   The inequality in (3) follows from the (merely) exponential growth of $l_{q+1}$.  The proof of (4) proceeds similarly to that of (2).  (5) follows from the super exponential growth of $\laq$ provided $\beta$ is small enough, and (6) follows from (5) provided $\alpha \ll \beta$. The final inequality follows from the definition of $\ell=\lambda_q^{-\frac{3}{4}}\lambda_{q+1}^{-\frac{1}{4}}$.
\end{proof}

\subsection{Inductive Step}
The proofs of \cref{maintheorem} and \cref{maintheorem2d} will require the following inductive propositions.
\begin{prop}[\textbf{3D QG Inductive Proposition}]\label{inductiveproposition}
Let $e(t):\mathbb{R}\rightarrow [0,\infty)$ be a smooth, compactly supported energy profile.  Then given $c>\frac{5}{2}$, there exists $b>1$, $a\gg 1$ such that the following holds.  Given a triple $\left( \nabla\Psi_q, \Mdot_q, Q_q \right)$ satisfying the inductive assumptions \eqref{inductiveequation}-\eqref{inductiveenergyprofiletwo} with parameters $\delta_q$, $\laq$, $l_{q}$ defined in terms of $a$, $b$, and $c$, there exists a new triple $\left( \nabla\Psi_{q+1}, \Mdot_{q+1}, Q_{q+1} \right)$ satisfying \eqref{inductiveequation}-\eqref{inductiveenergyprofiletwo} with $q$ replaced by $q+1$.
\end{prop}

\begin{prop}[\textbf{2D Euler Inductive Proposition}]\label{eulerinductiveproposition}
With the additional assumption that the matrix field $\Mdot_q$ is of the block form
$$ \Mdot_q = \left[ {\begin{array}{ccc}
m_1 & m_2 & 0\\
m_2 & -m_1 & 0\\
0 & 0 & 0\\
\end{array}} \right] $$
and the elimination of any restrictions on the spatial support, the outcome of \cref{inductiveproposition} can be achieved while simultaneously prescribing that $$\partial_z\left( \Psi_{q+1}-\Psi_q  \right) \equiv 0.$$
In particular, if $\Mdot_1$ is of such a block form, then one can impose that $\partial_z\Psi_q \equiv 0$ for all $q\in\mathbb{N}$.
\end{prop}

\section{Error Estimates}\label{Errorestimates}
Before estimating the transport, Nash, and oscillation errors, we show the following bounds on the perturbation and $\nabla\Psi_q$. Aside from the minor adjustment needed to prove (6) due to the spatial localizer, the following estimates all of course have exact analogues in previous convex integration arguments (for example Lemmas 4.3 and 4.4 from \cite{bsv16}). Any usage of the symbol $\lesssim$ indicates dependence only on universal constants (in particular not on $q$) and will be rectified later by a large choice of the parameter $a$.

\begin{lemma}[{\textbf{Preliminary Estimates}}]\label{help}
Using the definitions given in the previous section for each function and parameter, the following hold. 
\begin{enumerate}
    \item $\left\| \nabla^k \akl \right\|_{C^0(\supp \Xl)} \lesssim \dqone^\frac{1}{2}\laq \ell^{k-1}$ for $k\in\mathbb{N}$.
    \item For $t\in \supp \mathcal{X}_l$, $ \left\| D\Phi_l - \Id \right\|_{C^0} \lesssim \frac{\dq^\frac{1}{2}\laq}{\muqone}$ and $\left\| \nabla^N \Phi_l \right\|_{C^0} \lesssim \frac{\dq^\frac{1}{2}\laq^N}{\muqone}$ when $N\geq 2$.
    \item $\| \nabla e^{i\laqone(\Phi_l - x)\cdot k} \|_{C^0(\supp \Xl)} \lesssim \laqone^{1-\beta} $ and $\left\| \nabla^k e^{i\laqone(\Phi_l - x)\cdot k} \right\|_{C^0} \lesssim \laqone^{k\left(1-\beta\right)}$ for $k\in\mathbb{N}$.
    \item $\left\| \Dtq \left( \nabla \Psi_q \right) \right\|_{C^0} \lesssim \dq\laq $.
    \item $\|\wkl\|_{C^1(\supp \Xl)} \lesssim \dqone^\frac{1}{2}\laqone$.
    \item $\left\| \nabla\left( \Lq \W \right) \right\|_{C^n} \lesssim \dqone^\frac{1}{2}\laqone^n $.
\end{enumerate}
\end{lemma}

\begin{proof} 
\begin{enumerate} 
\item Using the chain rule estimates in \cref{chainrule}, we can write
\begin{align*}
    \left\| \nabla \akl \right\|_{C^0(\supp \mathcal{X}_l)} &= \left\| \nabla \left( \sqrt{\rho_l} c_{j,k}\left(\frac{M_{q,l}}{\rho_l}\right) \right) \right\|_{C^0}\\
    &\leq \sqrt{\rho_l}\left( \| \nabla c_{j,k} \|_{C^0} \| \nabla M_{q,l} \|_{C^0} \rho_l^{-1} \right).
\end{align*}
We first note that either $M_{q,l}$ is constant or $\rho_l \geq \frac{\dqone}{16}$ (observed in \eqref{eitheror}), and thus we can reduce to the case that $\rho_l \geq \frac{\dqone}{16}$. Then to estimate $\nabla M_{q,l}$, we use that $\M_{q,l}$ solves the equation
$$  \Dtq \M_{q,l} = \partial_t \M_{q,l} + \grad^\perp \Psi_{q,l} \cdot \grad \M_{q,l} = 0.  $$
Then the transport estimates \cref{transport}, the inductive assumptions \eqref{inductivevelocity} and \eqref{inductiveerror}, and the parameter assumptions from \cref{parameterinequalities} yield
\begin{align*}
    \left\| \nabla \akl \right\|_{C^0(\supp \mathcal{X}_l)}
&\leq C(c_{j,k}) \sqrt{\rho_l} \left( \frac{\eta \dqone \laq}{\rho_l} \right)\\
    &\lesssim \dqone^\frac{1}{2}\laq.
\end{align*}
For the second bound, arguing as before and using the $C^k$ bounds on $\nabla\Psi_q$ and $\M_{q,\ell}$, and therefore $\Mdot_{q,l}$ and $M_{q,l}$, gives the claim.
\item Applying \cref{chainrule} and the transport estimates in \cref{transport} , we have that 
$$ \| D\Phi_l - \Id \|_{C^0} \leq (t-t_0)\|\grad\grad^\perp\Psi_q\|_{C^0}e^{(t-t_0)\|\grad\grad^\perp \Psi_q\|_{C^0}} \leq \frac{\dq^\frac{1}{2}\laq}{\muqone}. $$
The last estimate follows again from \cref{transport} and the $C^n$ bounds of the velocity $\grad^\perp\Psi_q$. 
\item We can use (2) and \cref{chainrule} to calculate
\begin{align*}
    \left\| \nabla \left( e^{i\laqone (\Phi_l(x,t)-x)\cdot k} \right) \right\|_{C^0({\supp \mathcal{X}_l})} &\leq \left( \| \nabla e^{ix} \|_{C^0} \left\| \nabla\left(i\laqone k\cdot (\Phi_l -x)\right) \right\|_{C^0} \right) \\
    &\leq \laqone \left\| D\Phi_l - \Id \right\|_{C^0} \\
    &\leq \laqone \frac{\dq^\frac{1}{2}\laq}{\muqone}\\
    &\leq \laqone^{1-\beta}.
\end{align*}
The second claim follows from the $C^n$ bounds on $\grad^\perp\Psi_q$, the chain rule \cref{chainrule}, and the transport estimates \cref{transport}.
\item We have that $\nabla\Psi_q$ satisfies the transport equation
$$ \partial_t(\nabla\Psi_q) + \grad^\perp\Psi_q \cdot \grad \nabla\Psi_q = \grad Q_{3,p} + \grad\cdot \M_q. $$
By the inductive assumptions \eqref{inductiveerror} and \eqref{inductivecurl}, 
$$ \| \grad Q_{3,p} \| \leq \dq \laq , \qquad \| M_q \|_{C^1} \leq 4 \eta \dqone \laq $$
which yields the claim since $\dqone \leq \dq$.
\item Using that $\Dtq\wkl=0$ and that $\wkl = \akl \expk ik$ at $t=\frac{l}{\muqone}$, we apply \cref{transport} to obtain that
$$ \|\wkl\|_{C^1} \leq \left( \|\akl\|_{C^1} + \|\akl\|_{C^0} \laqone \right)e^{\frac{\|\grad^\perp\Psi_q\|_{C^1}}{\muqone}} \leq \dqone^\frac{1}{2} \laq + \dqone^\frac{1}{2}\laqone, $$
proving the result.
\item Applying the Leibniz rule to $\nabla\left( \Lq \W \right)$, using the compact frequency support of $\nabla\W$, and noticing that $\nabla^n \Lq = (\partial_z)^n \Lq \ll \laqone^n$ due to the fact that $l_{q+2} \ll \laqone$ gives the claim.
\end{enumerate}

\end{proof}

\subsection{Transport Error}
\begin{lemma}\label{transporterror}
The transport error
$$ \partial_t\left(\pert\right) + \grad^\perp \Psi_q \cdot \grad \pert $$
is equal to
$$ \curl\left({Q}_{T}\right) + \grad \cdot \M_{T} $$
with the estimates
$$ \| {Q}_{T} \|_{C^0} \leq \dqone, \qquad \| {Q}_{T} \|_{C^1} \leq \dqone \laqone $$
$$ \| \M_{T} \|_{C^0} \leq \eta\dqtwo, \qquad \| \M_T \|_{C^1} \leq \dqtwo \laqone, \qquad \| \Dtq \M_T \|_{C^0} \leq \dqtwo \dqone^\frac{1}{2} \laqone. $$
Furthermore, $Q_{T}$ and $\M_{T}$ are supported in the set $$\mathbb{T}^2 \times \left[ \frac{1}{l_{q+1}}, 2\pi - \frac{1}{l_{q+1}} \right].$$
\end{lemma}

\begin{proof}
By the compact support in $x$ and $y$ frequency modes of $\nabla\Psi_q$ and the support in frequency of $\nabla\left(L\W\right)$ in a cylinder whose base is an annulus in $x$ and $y$ centered around $\laqone$, the $x$ and $y$ frequency modes of the transport error are supported in the cylinder above an annulus of radius $\laqone$ in $\mathbb{Z}^2$.  Therefore, we can apply the $x$ and $y$ frequency localizer $\Pbarqone$ and \cref{flatfrequencylocalizers} to write the transport error as
\begin{align*}
\Pbarqone &\left( \partial_t(\pert) + \grad^\perp \Psi_q \cdot \grad \pert \right)\\
&= \Pbarqone \left(\Lq \Dtq \nabla \W + \left( 0, 0, \partial_z \Lq  \partial_t \W \right)^t \right)\\
&:= M_{T,1} + M_{T,2}.
\end{align*}
Beginning with $M_{T,1}$, we can commute $\Lq$ and $\Pbarqone$ and introduce the commutator $\left[\Dtq, \Pgradk\right]$ to write
\begin{align*}
    \left\| M_{T,1} \right\|_{C^0} &\leq \left\| \Pbarqone  \left( \left[ \Dtq, \Pgradk \right] (\Xl\wkl) + \partial_t \Xl \wkl \right) \right\|_{C^0}  \\
    &\leq \left\| \Pbarqone \sum_{k,l}\left[ \grad^\perp\Psi_q \cdot \grad, \Pgradk \right](\wkl \Xl)\right\|_{C^0} + \left\| \Pbarqone \sum_{k,l} \Pgradk\left(\partial_t \Xl \wkl\right) \right\|_{C^0} \\
    &\lesssim  \left\| \grad^\perp\Psi_q \right\|_{C^1} \left\| \wkl \Xl \right\|_{C^0} + \left\| \partial_t \Xl \akl e^{i\laqone \Phi_l \cdot x} \right\|_{C^0}  \\
    &\lesssim  \dq^\frac{1}{2} \laq \dqone^\frac{1}{2} + \muqone \dqone^\frac{1}{2} \\
    &\lesssim \muqone \dqone^\frac{1}{2}\\
    &\leq \eta \dqtwo \laqone
\end{align*}
after applying the commutator estimate \eqref{commutatorbsv1}. We then decompose $M_{T,1}$ using $\Pgradbar$ and $\Pgradbarperp$ as 
$$ M_{T,1} = \Pgradbar \left( M_{T,1} \right) + \Pgradbarperp\left( M_{T,1} \right). $$
After applying $\D$, we can absorb the first piece into $\Mdot_T$, while the second piece becomes part of $Q_T$ after inverting $\grad^\perp$. The desired $C^0$ bounds on $M_T$ and $Q_T$ follow from the presence of the frequency localizer $\Pbarqone$, the fact that $\D$ and $\left(\grad^\perp\right)^{-1}$ are operators of order $-1$ in $x$ and $y$, and an application of \cref{bernstein}. To show the $C^1$ bounds, we write 
\begin{align*}
    \nabla \left( \left[ \grad^\perp\Psi_q \cdot \grad, \Pgradk \right]\left(\Xl \wkl\right) \right) &= \left[ \nabla\grad^\perp\Psi_q \cdot \grad, \Pgradk \right](\Xl\wkl) + \left[ \grad^\perp\Psi_q \cdot \grad, \Pgradk \right] \left( \nabla \wkl \Xl\right).
\end{align*}

\begin{align*}
    \left\| \nabla M_{T,1} \right\|_{C^0} &\leq \left\| \nabla \Lq \Dtq \nabla\W \right\|_{C^0} + \left\| \Lq \nabla \left( \Dtq \nabla\W \right) \right\|_{C^0}\\
    &\lesssim \left\| \partial_z \Lq \right\|_{C^0} \left\| \Dtq \nabla\W \right\|_{C^0} + \left\| \left[ \nabla\grad^\perp\Psi_q \cdot \grad, \Pgradk \right](\Xl\wkl) \right\|_{C^0}\\
    &\qquad + \left\| \left[ \grad^\perp\Psi_q \cdot \grad, \Pgradk \right] \left( \nabla \wkl \Xl\right) \right\|_{C^0} + \left\| \Pgradk \left( \partial_t \Xl \nabla\wkl \right) \right\|_{C^0}\\
    &\leq l_{q+1} \frac{\dqtwo^2}{\dqone}\laqone + \left\| \grad^\perp\Psi_q \right\|_{C^2} \| \Xl \wkl \|_{C^0} + \| \grad^\perp\Psi_q \|_{C^1} \| \Xl \wkl \|_{C^1} + \left\| \Pgradk \left( \partial_t \Xl \nabla\wkl \right) \right\|_{C^0}\\
    &\leq  \dqtwo \laqone + \dq^\frac{1}{2}\laq^2 \dqone^\frac{1}{2} + \dq^\frac{1}{2}\laq\dqone^\frac{1}{2}\laqone + \muqone \dqone^\frac{1}{2}\laqone\\
    &\leq \dqtwo \laqone^2.
\end{align*}
Then using the fact that differentiating and multiplying by $\Lq$ or $\nabla\Lq$ commutes with $\D$ and $\left(\grad^\perp\right)^{-1}$ and applying \cref{bernstein} due to the $x$ and $y$ frequency support allows us to divide by $\laqone$, proving the claim. The spatial support of each term is satisfactory using the inductive hypothesis \eqref{inductivespatialsupport} and the fact that multiplication by $\Lq$ commutes with convolution operators in $x$ and $y$.

The entirety of the second term $M_{T,2}$ will be absorbed into $\Mdot_T$ by applying $\D$. Since multiplication by $\Lq$ and $\partial_z\Lq$ commutes with $\D\Pbarqone$, we have that 
\begin{align*}
    \| M_{T,2} \|_{C^0} \lesssim \frac{1}{\laqone} \| \partial_z \Lq \|_{C^0} \| \partial_t \W \|_{C^0}.
\end{align*}
Since $\partial_t \nabla \W = \Dtq \nabla \W - \grad^\perp\Psi_q \cdot \grad \nabla \W$, we have that 
$$ \| \partial_t \nabla \W \|_{C^0} \lesssim \muqone \dqone^\frac{1}{2} + \dqone^\frac{1}{2} \laqone \lesssim \dqone^\frac{1}{2} \laqone. $$
Noticing that $\W=(-\Delta)^{-1}\nabla\cdot\left(\nabla\W\right)$ and using the frequency support of $\W$, we can apply \cref{bernstein} to obtain
$$ \| \partial_t \Wqone \|_{C^0} \lesssim \dqone^\frac{1}{2}. $$
Plugging in this estimate, we obtain
\begin{align*}
    \|M_{T,2}\|_{C^0} &\lesssim \frac{1}{\laqone} l_{q+1} \dqone^\frac{1}{2}\\
    &\leq \eta \dqtwo.
\end{align*}
The $C^1$ bound follows from estimating
\begin{align*}
\left\| \partial_z \Lq \partial_t \W \right\|_{C^1} &\leq \|\partial_z \Lq \|_{C^1} \|\partial_t \W\|_{C^0} + \|\partial_z \Lq \|_{C^0} \|\partial_t \W\|_{C^1} \\
&\leq l_{q+1}^2 \dqone^\frac{1}{2} + l_{q+1} \dqone^\frac{1}{2}\laqone
\end{align*}
applying $\D$, using the frequency support in $x$ and $y$ to divide by a factor of $\laqone$, and recalling that $l_{q+1} \leq \dqone^\frac{1}{2}\laqone$.

Before beginning to estimate the material derivative $\Dtq \Mdot_T$, note that $\Dtq \Lq = 0$. The material derivative of the transport error can then be decomposed as
\begin{align*}
\Dtq &\left( \D \circ \Pbarqone \bigg{(} \partial_t(\pert) + \grad^\perp\Psi_q \cdot \grad (\pert) \bigg{)} \right) = \\
&\quad \Lq \left[ \Dtq, \D \Pbarqone  \right] \left( \Dtq (\nabla\W) \right)\\
&\qquad + \Lq \D \Pbarqone \left( \Dtq \left( \sum_{kl}\Pgradk \left( \partial_t \Xl \wkl \right) \right) \right) \\
&\qquad + \Lq \D \Pbarqone \left( \Dtq \left( \sum_{kl} \left[ \Dtq, \Pgradk \right] (\Xl \wkl) \right) \right) \\
&\qquad + \partial_z \Lq \Dtq \left( \D\Pbarqone (0,0,\partial_t \W)^t \right) \\
&:= T_1 + T_2 + T_3 + T_4.
\end{align*}

Beginning with $T_1$, we have that by the commutator estimate \eqref{commutatorbsv1} and the estimate on the amplitude given above,
\begin{align*}
    \| T_1 \|_{C^0} &\lesssim \frac{1}{\laqone} \| \grad^\perp\Psi_q \|_{C^1} \| \Dtq(\nabla\W) \|_{C^0}\\
    & \lesssim \frac{1}{\laqone} \dq^\frac{1}{2} \laq  \muqone\dqone^\frac{1}{2}.
\end{align*}
Using that $\dq^\frac{1}{2}\laq \leq \muqone$ and $\frac{\muqone^2}{\laqone}\dqone^\frac{1}{2} \leq \dqtwo\dqone^\frac{1}{2}\laqone$, we obtain
$$ \| T_1 \|_{C^0} \leq \dqtwo \dqone^\frac{1}{2}\laqone. $$
Moving on to $T_2$, we apply \eqref{commutatorbsv1} and estimate the parameters as in $T_1$ to obtain
\begin{align*}
    \left\| T_2 \right\|_{C^0} &= \left\| \D \Pbarqone \left[ \Dtq \left( \sum_{kl} \Pgradk \left( \partial_t \Xl \wkl \right) \right) \right] \right\|_{C^0} \\
    &= \left\| \D \Pbarqone \left[ \sum_{kl} \left[ \Dtq,\Pgradk \right]\left( \partial_t \Xl \wkl \right) + \sum_{kl}\Pgradk \left( \partial_t^2 \Xl \wkl \right) \right] \right\|_{C^0} \\
    &\lesssim \frac{1}{\laqone} \left( \|\grad^\perp\Psi_q\|_{\Cbar^1}\|\partial_t\Xl\wkl\|_{C^0} + \|\partial_t^2 \Xl \wkl \|_{C^0} \right)\\
    &\lesssim \frac{1}{\laqone} \left( \dq^\frac{1}{2}\laq \muqone \dqone^\frac{1}{2} + \muqone^2 \dqone^\frac{1}{2} \right)\\
    &\lesssim \frac{1}{\laqone} \muqone^2 \dqone^\frac{1}{2}\\
    &\leq \dqtwo\dqone^\frac{1}{2}\laqone.
\end{align*}

We now estimate the material derivative of $T_3$. As everything is localized in $x$ and $y$ frequencies in an annulus of radius $\laqone$, we estimate the terms inside parentheses directly and then divide by $\frac{1}{\laqone}$ at the end.  We write
\begin{align*}
    \Dtq \left( \sum_{kl} \left[ \Dtq,\Pgradk \right](\Xl\wkl) \right) &= \left[ \Dtq, \sum_{kl} \left[ \Dtq, \Pgradk \right] \right] (\Xl\wkl) + \left[ \Dtq, \Pgradk \right] \left(\Dtq \left(\Xl\wkl\right)\right)\\
    &\qquad =: T_{3,1} + T_{3,2}.
\end{align*}
We can estimate $T_{3,2}$ using the commutator estimate \eqref{commutatorbsv1} as
\begin{align*}
    T_{3,2} &\leq \| \nabla\Psi_q \|_{C^1} \| \partial_t \Xl \wkl \|_{C^0}\\
    &\leq \dq^\frac{1}{2} \laq \muqone \dqone^\frac{1}{2}\\
    &\leq \muqone^2 \dqone^\frac{1}{2}.
\end{align*}
Applying $\D$ and dividing by $\laqone$ gives the desired estimate. For $T_{3,1}$, we apply the iterated commutator estimate \eqref{iteratedcommutator} to obtain
\begin{align*}
    T_{3,1} &\leq \frac{1}{\laqone} \left\|\nabla\Psi_q\right\|_{C^1}^2 \| \Xl \wkl \|_{C^1} + \|\Xl\wkl\|_{C^0} \left( \laqone \left\|\Dtq \grad^\perp\Psi_q\right\|_{C^0} + \left\|\grad^\perp\Psi_q\right\|_{C^1}^2 \right) \\
    &\leq \frac{1}{\laqone} \dq \laq^2 \dqone^\frac{1}{2}\laqone + \dqone^\frac{1}{2} \left( \laqone \dq \laq + \dq \laq^2 \right)\\
    &\lesssim \dq \dqone^\frac{1}{2} \laq \laqone.
\end{align*}
Applying $\D$ and dividing again by $\laqone$, we obtain the desired estimate.

Finally, we write $T_4$ as
\begin{align*}
    T_4 = \partial_z \Lq \left[ \Dtq,\D\Pbarqone \right](\partial_t \W) + \partial_z \Lq \D\Pbarqone\left( \Dtq (\partial_t \W) \right).
\end{align*}
We can estimate the first term using the commutator estimate \eqref{commutatorbsv1} by $$l_{q+1}\frac{1}{\laqone} \dq^\frac{1}{2} \laq \dqone^\frac{1}{2} \leq \dqtwo \dqone^\frac{1}{2}\laqone $$
as desired. For the second term, first note that
$$ \partial_t \W = \Dtq \W - \grad^\perp\Psi_q\cdot \grad\W. $$
Handling the second piece of the second term first, we then have that
\begin{align*}
    &\left\| \partial_z \Lq \D\Pbarqone \left( \Dtq\left( \grad^\perp\Psi_q \cdot \grad\W \right) \right) \right\|_{C^0} \\
    &\qquad\leq \left\| \partial_z \Lq \D\Pbarqone \left( \Dtq\left(\grad^\perp\Psi_q\right) \cdot \grad \W \right) \right\|_{C^0}\\
    &\qquad\qquad + \left\| \partial_z \Lq \D\Pbarqone \left( \grad^\perp\Psi_q \cdot \Dtq\left(\grad \W\right) \right) \right\|_{C^0}\\
    &\lesssim l_{q+1} \frac{1}{\laqone} \left( \dq \laq \dqone^\frac{1}{2} + \muqone\dqone^\frac{1}{2} \right)\\
    &\leq \dqtwo \dqone^\frac{1}{2} \laqone.
\end{align*}

Before beginning to estimate the first piece of the second term, note that
$$ \W = (-\Delta)^{-1} \left( \nabla \cdot \left( \Pgradk \Xl\wkl \right) \right). $$
Denoting the operator $(-\Delta)^{-1}\circ (\nabla\cdot)\circ \Pgradk$ by $K$, we have that $K$ is an order $-1$ convolution kernel.  Therefore, we can write that
\begin{align*}
    \partial_z \Lq \D\Pbarqone &\left(  \Dtq \left( \Dtq \left( \W \right) \right) \right) \\
    &= \partial_z \Lq \D\Pbarqone \left(  \Dtq \left[ \Dtq, K \right] (\Xl\wkl) \right) + \partial_z \Lq \D\Pbarqone \left( \Dtq \left( K(\partial_t\Xl\wkl) \right) \right).
\end{align*}
The second term is bounded as follows:
\begin{align*}
    \left\| \partial_z \Lq \D\Pbarqone \left( \Dtq \left( K (\partial_t\Xl\wkl) \right)\right) \right\|_{C^0} &\leq \left\| \partial_z \Lq \D\Pbarqone \left( \left[ \Dtq, K \right] (\partial_t \Xl \wkl) \right) \right\|_{C^0} \\
    &\qquad + \left\| \partial_z \Lq \D\Pbarqone K(\partial_t^2 \Xl \wkl) \right\|_{C^0}\\
    &\lesssim l_{q+1} \frac{1}{\laqone}\frac{1}{\laqone} \dq^\frac{1}{2} \laq \muqone \dqone^\frac{1}{2} + l_{q+1} \frac{1}{\laqone}\frac{1}{\laqone}\muqone^2 \dqone^\frac{1}{2}\\
    &\leq \dqtwo \dqone^\frac{1}{2}\laqone.
\end{align*}
Here we have used the presence of $\Pbarqone$ and \cref{flatfrequencylocalizers} to see that $K$ gains a factor of $\frac{1}{\laqone}$. Then for the first term, we will use the iterated commutator estimate \eqref{iteratedcommutator} again.  We can then write
\begin{align*}
    &\left\| \partial_z \Lq \D\Pbarqone \left(  \Dtq \left[ \Dtq, K \right] (\Xl\wkl) \right) \right\|_{C^0} \leq \left\| \partial_z \Lq \D\Pbarqone \left( \left[ \Dtq ,\left[ \Dtq, K \right] \right] (\Xl\wkl) \right) \right\|_{C^0} \\
    &\qquad\qquad\qquad\qquad + \left\|  \partial_z \Lq \D\Pbarqone \left( \left[ \Dtq, K \right] (\partial_t \Xl \wkl) \right) \right\|_{C^0}\\
    &\qquad \leq \left\| \partial_z \Lq \right\|_{C^0}\laqone^{-1} \left( \laqone^{-2} \left\| \nabla\Psi_q \right\|_{C^1}^2 \|\Xl\wkl\|_{C^1} + \|\Xl\wkl\|_{C^0}\left( \left\| \Dtq \nabla\Psi_q \right\| + \laqone^{-1}\|\nabla\Psi_q\|_{C^1}^2 \right) \right)\\
    & \qquad \qquad + \| \partial_z \Lq \|_{C^0} \laqone^{-1} \laqone^{-1} \|\nabla\Psi_q\|_{C^1} \|\partial_t \Xl \wkl\|_{C^0} \\
    &\qquad \leq l_{q+1} \frac{1}{\laqone} \left( \laqone^{-2}(\dq^\frac{1}{2}\laq)^2\dqone^\frac{1}{2}\laqone + \dqone^\frac{1}{2}\left(\dq\laq + \laqone^{-1} (\dq^\frac{1}{2}\laq)^2\right) \right)\\
    &\qquad \qquad + l_{q+1} \frac{1}{\laqone^2}\dq^\frac{1}{2}\laq\muqone\dqone^\frac{1}{2}\\
    &\qquad \leq \dqtwo\dqone^\frac{1}{2}\laqone,
\end{align*}
concluding the proof.
\end{proof}

\subsection{Nash Error}

\begin{lemma}\label{nasherror}
The Nash error
$$ \grad \cdot \left( \grad^\perp \left( \Lq\W \right) \otimes \nabla \Psi_q \right) $$
is equal to
$$ \curl\left({Q}_{N}\right) + \grad \cdot \M_{N} $$
with the estimates
$$ \| {Q}_{N} \|_{C^0} \leq \dqone, \qquad \| {Q}_{N} \|_{C^1} \leq \dqone \laqone $$
$$ \| \M_{N} \|_{C^0} \leq \eta\dqtwo, \qquad \| \M_N \|_{C^1} \leq \dqtwo \laqone, \qquad \| \Dtq \M_N \|_{C^0} \leq \dqtwo \dqone^\frac{1}{2} \laqone. $$
Furthermore, $Q_{N}$ and $\M_{N}$ are supported in the set $$\mathbb{T}^2 \times \left[ \frac{1}{l_{q+1}}, 2\pi - \frac{1}{l_{q+1}} \right].$$
\end{lemma}

\begin{proof}
Due to the spatial support of $\nabla\Psi_q$ and $\grad^\perp(\mathbb{W}_{q+1}L_{q+1})$, the Nash error is equal to 
$$\grad \cdot \left( \nabla\Psi_q \otimes \grad^\perp\W \right)$$
and the claim on the spatial support is immediate since we shall only ever convolve in $x$ and $y$. We calculate the amplitude by writing 
\begin{align*}
    \left\|  \grad \cdot \left( \nabla\Psi_q \otimes \grad^\perp\W \right) \right\|_{C^0} &\leq \left\| \grad(\nabla\Psi_q) \grad^\perp\W \right\|_{C^0}\\
    &\leq \dq^\frac{1}{2} \laq \dqone^\frac{1}{2} \\
    &\leq \eta \dqtwo \laqone.
\end{align*}
Decomposing into $\Pgradbar$ and $\Pgradbarperp$ and using Bernstein's inequality as for the transport error shows the desired $C^0$ bounds on $Q_N$ and $\Mdot_N$. The $C^1$ bounds follow by applying $\nabla$ to the Nash error and noticing that the $x$ and $y$ frequency support $\grad^\perp\W\cdot\grad\nabla\Psi_q$ is contained in an annulus of radius $\laqone$, allowing us to divide by $\laqone$ after applying $\D$ and $\left(\grad^\perp\right)^{-1}$.

Moving now to the material derivative, we use \eqref{commutatorbsv1} to write that
\begin{align*}
    &\left\| \Dtq \left( \D \Pbarqone \grad \cdot \left( \nabla\Psi_q \otimes \grad^\perp\W \right) \right) \right\|_{C^0} \leq \left\|  \D \Pbarqone \left( \Dtq \left(   \grad^\perp\W \cdot\grad\nabla\Psi_q  \right)\right) \right\|_{C^0}\\
    &\qquad \qquad \qquad \qquad + \left\|   \left[ \D \Pbarqone, \Dtq \right] \left( \grad^\perp\W \cdot\grad\nabla\Psi_q  \right) \right\|_{C^0}\\
    &\qquad \qquad \lesssim \frac{1}{\laqone} \left( \left\| \Dtq \left( \grad^\perp\W \cdot \grad \nabla\Psi_q \right) \right\|_{C^0} + \left\|\grad^\perp\Psi_q\right\|_{C^1}\left\|\grad^\perp\W\cdot\grad\nabla\Psi_q\right\|_{C^0} \right)\\
    &\qquad \qquad \leq \frac{1}{\laqone} \bigg{(} \left\|\Dtq\grad^\perp\W\right\|_{C^0}\left\|\grad\nabla\Psi_q\right\|_{C^0} + \left\|\grad^\perp\W\right\|_{C^0} \left\|\Dtq(\grad\nabla\Psi_q)\right\|_{C^0}\\
    &\qquad \qquad \qquad + \left\|\grad^\perp\Psi_q\right\|_{C^1}\left\|\grad^\perp\W\cdot\grad\nabla\Psi_q\right\|_{C^0}  \bigg{)}\\
     &\qquad \qquad \leq \frac{1}{\laqone} \bigg{(} \left\|\Dtq\grad^\perp\W\right\|_{C^0}\left\|\grad\nabla\Psi_q\right\|_{C^0} + \left\|\grad^\perp\W\right\|_{C^0} \left\|\grad\Dtq(\nabla\Psi_q)\right\|_{C^0}\\
    &\qquad \qquad \qquad + \left\| \grad^\perp\W \right\|_{C^0} \left\| \nabla\Psi_q \right\|_{C^1}^2 +   \left\|\grad^\perp\Psi_q\right\|_{C^1}\left\|\grad^\perp\W\cdot\grad\nabla\Psi_q\right\|_{C^0}  \bigg{)}\\   
    &\qquad \qquad \leq \frac{1}{\laqone} \left( \muqone \dqone^\frac{1}{2} \dq^\frac{1}{2} \laq + \dqone^\frac{1}{2} \dq \laq^2 + \dqone^\frac{1}{2} \dq \laq^2 + \dqone^\frac{1}{2} \dq \laq^2 + \dq^\frac{1}{2} \laq \dqone^\frac{1}{2} \dq^\frac{1}{2}\laq \right)\\
    &\qquad \qquad \lesssim \frac{1}{\laqone}\muqone^2 \dqone^\frac{1}{2}\\
    &\leq \dqtwo \dqone^\frac{1}{2} \laqone.
\end{align*}
\end{proof}

\subsection{Oscillation Error}
Before defining and estimating the oscillation error, we address the effect of the localizer $\Lq$.  
As discussed earlier, $\Lq$ \textit{factors out} of the oscillation error. The interaction of the perturbation $\pert$ with itself is given in the term
\begin{align*}
    \grad \cdot \left( \nabla(\Lq\W) \otimes \grad^\perp(\Lq \W) \right). 
\end{align*}
Since $L_q$ depends only on $z$, the first two components are equal to
$$ \Lq^2 \grad \cdot \left( \grad\W \otimes \grad^\perp \W \right). $$
In the third row, we can write that
\begin{align*}
     \grad \cdot \left(\grad^\perp\left(\Lq\W\right) \partial_z \left(\Lq\W\right) \right) &= \grad \cdot \left( \grad^\perp\left(\Lq\W\right) \left(\W \partial_z \Lq + \Lq \partial_z \W \right) \right)\\
    &= \Lq \partial_z \Lq  \grad^\perp\W \cdot \grad \W  + \Lq^2  \grad^\perp \W \cdot \grad(\partial_z \W)\\
    &= \Lq^2 \grad^\perp\W \cdot \grad \partial_z \W,
\end{align*}
showing that
\begin{align}\label{importantcancellation}
    \grad \cdot \left( \nabla (\Lq \W ) \otimes \grad^\perp (\Lq \W ) \right) = \Lq^2 \grad \cdot \left( \nabla \W \otimes \grad^\perp \W \right).
\end{align}
By the inductive assumption \eqref{inductivespatialsupport} on the spatial support of $\M_q$, we have also that
$$\grad\cdot\M_q = \Lq^2 \grad\cdot\M_q.$$ Therefore 
\begin{align}\label{factoringout}
    \grad \cdot \left( \nabla (\Lq \W ) \otimes \grad^\perp (\Lq \W ) \right) + \grad \cdot \M_q = \Lq^2 \grad \cdot \left( \nabla\W \otimes \grad^\perp \W + \M_q \right).
\end{align}
We will decompose the right hand side into several terms.  The definition of this decomposition as well as the estimates for each piece comprise the remainder of this section. We first collect some preliminary estimates.

\begin{lemma}\label{olowstuff}
The following estimates hold.
\begin{enumerate}
\item For $\theta\in[0,1]$, $\|\wkl\|_{C^\theta} \lesssim \dqone^\frac{1}{2}\laqone^\theta$.
\item For $\theta\in[0,2]$, $\left\| \left[ \Pgradk, \aklexp \right](\expk ik) \right\|_{{C}^\theta} \lesssim \dqone^\frac{1}{2} \laqone^{\theta-\beta}$.
\item $\left\| \Dtq \left( \left[ \Pgradk , \aklexp \right] \left( \expk ik \right) \right) \right\|_{C^0} \lesssim \muqone \dqone^\frac{1}{2}$.
\end{enumerate}
\end{lemma}
\begin{proof}
The proof of (1) follows from interpolating
$$\|\wkl\|_{C^0}\leq\dqone^\frac{1}{2}, \qquad \|\wkl\|_{C^1}\leq\dqone^\frac{1}{2}\laqone$$
using \cref{help} and \cref{holder spaces}. To prove (2), recall that by \cref{help}, each derivative on $\aklexp$ costs a factor of $\laqone^{1-\beta}$. Then we can apply the commutator estimate \eqref{commutatorbsv2} to obtain
\begin{align*}
    \left\| \left[ \Pgradk, \aklexp \right](\expk ik) \right\|_{C^k} &\lesssim \frac{1}{\laqone} \sum_{0\leq j \leq k} \left\|\nabla\aklexp\right\|_{{C}^j} \left\|e^{i\laqone k\cdot x} ik\right\|_{{C}^{k-j}} \\
    &\lesssim \dqone^\frac{1}{2}\laqone^{k-\beta}.
\end{align*}
The non-integer bounds then follow from interpolation.  To prove (3), observe that 
$$ \left[ \Pgradk , \aklexp \right] \left( \expk ik \right) = \Dtq \left( \Pgradk(\wkl) \right) - \Dtq \wkl. $$
and use the estimates in the section on the transport error.
\end{proof}

\subsubsection{Estimates for \texorpdfstring{O\textsubscript{high}}{O high}}

\begin{lemma}\label{oscillationhigh}
The high frequency portion of the oscillation error
$$\Lq^2 \grad \cdot \left( \sum_{k+k'\neq 0} \Xl\Xlprime \left(  \Pgradk\left(\wkl\right) \right)\otimes \left( \Pgradkbarperpprime\left(\overline{\wklprime}^\perp\right)\right) \right)$$
is equal to
$$ \curl\left({Q}_{high}\right) + \grad \cdot \Mdot_{O,high} $$
with the estimates
$$ \| {Q}_{high} \|_{C^0} \leq \dqone, \qquad \| {Q}_{high} \|_{C^1} \leq \dqone \laqone $$
$$ \| \Mdot_{O,high} \|_{C^0} \leq \eta\dqtwo, \qquad \| \Mdot_{O,high} \|_{C^1} \leq \dqtwo \laqone, \qquad \| \Dtq \Mdot_{O,high} \|_{C^0} \leq \dqtwo \dqone^\frac{1}{2} \laqone. $$
Furthermore, $Q_{high}$ and $\grad\cdot \Mdot_{O,high}$ are supported in the set $$\mathbb{T}^2 \times \left[ \frac{1}{l_{q+1}}, 2\pi - \frac{1}{l_{q+1}} \right].$$
\end{lemma}

\begin{proof}

Towards obtaining a decomposition, we can apply the frequency localizer $\overline{\mathbb{P}}_{\approx \laqone}$ since $k\neq k'$ and \cref{flatfrequencylocalizers} to write
\begin{align*}
    &\Lq^2 \grad \cdot \sum_{k+k'\neq 0} \Xl\Xlprime \Pgradk\left(\wkl\right) \otimes  \Pgradkbarperpprime\left(\overline{\wklprime}^\perp\right) \\
    &= \Lq^2 \grad \cdot \overline{\mathbb{P}}_{\approx \laqone} \sum_{k+k'\neq 0} \Xl \Xlprime \bigg{(} \left( \left[ \Pgradk, \aklexp \right](\expk ik) + \wkl \right) \otimes\\
    &\qquad \qquad \qquad \qquad \left( \left[ \Pgradkbarperpprime, \aklexpprime \right](\expkprime i\kbarperpprime) + \overline{\wklprime}^\perp\right) \bigg{)}\\
    &= \Lq^2 \grad \cdot \overline{\mathbb{P}}_{\approx \laqone} \sum_{k+k'\neq 0} \Xl \Xlprime \bigg{(} \left( \left[ \Pgradk, \aklexp \right](\expk ik)\right) \otimes \left( \Pgradkbarperpprime\left(\overline{\wklprime}^\perp\right) \right) \bigg{)}\\
    &\qquad  +\Lq^2 \grad \cdot \overline{\mathbb{P}}_{\approx \laqone} \sum_{k+k'\neq 0} \Xl \Xlprime \bigg{(} \left(  \wkl \right) \otimes \left( \left[ \Pgradkbarperpprime, \aklexpprime \right](\expkprime i\kbarperpprime)\right) \bigg{)}\\
    &\qquad  +\Lq^2 \grad\cdot \overline{\mathbb{P}}_{\approx \laqone} \sum_{k+k'\neq 0} \Xl \Xlprime \bigg{(} \left( \wkl \right) \otimes \left( \overline{\wklprime}^\perp\right) \bigg{)}\\
    &:= \Lq^2 \grad \cdot \left( O_{high,1} + O_{high,2} + O_{high,3}\right)
\end{align*}
The terms $O_{high,1}$ and $O_{high,2}$ are simpler to analyze, while the analysis of $O_{high,3}$ is more delicate and will be separated into its own lemma. 

Calculating the amplitude of $O_{high,1}$ and $O_{high,2}$, we apply \cref{olowstuff} to see that
\begin{align*}
    \left\| O_{high,1} \right\|_{C^0} + \left\| O_{high,2} \right\|_{C^0} &\lesssim \left\| \left[ \Pgradk, \aklexp \right](\expk ik) \right\|_{C^0} \left\|\Pgradkbarperpprime\left(\overline{\wklprime}^\perp \right) \right\|_{C^0} \\
    &\qquad\qquad + \| \wkl \|_{C^0} \left\| \left[ \Pgradkbarperpprime, \aklexpprime \right](\expkprime i\kbarperpprime) \right\|_{C^0}\\
    &\lesssim \dqone \laqone^{-\beta}\\
    &\leq \eta\dqtwo.
\end{align*}
Then we separate $\grad \cdot O_{high,1}$ and $\grad \cdot O_{high,2}$ using the projection operators $\Pgradbar$ and $\Pgradbarperp$ as
\begin{align*}
    \Lq^2 \grad \cdot \left( O_{high,1} + O_{high,2} \right) = \Lq^2 \left( \Pgradbar \left( \grad \cdot (O_{high,1} + O_{high,2}) \right) + \Pgradbarperp \left( \grad \cdot (O_{high,1} + O_{high,2}) \right) \right).
\end{align*}
Since applying $\Pgradbar$ gives a vector field with three components, the first two of which are the horizontal gradient $\grad$ of a scalar function, the first term can be plugged into the inverse divergence $\D$ and absorbed in $\Mdot_{O,high}$. Applying $\Pgradbarperp$ yields a vector field with no third component whose first two components are the perpendicular gradient $\grad^\perp$ of a scalar function, and so we absorb this term into $\curl(Q_{high})$.  Since multiplication by $\Lq$ commutes with both operators, the claims on the spatial supports of $Q_{high}$ and $\Mdot_{O,high}$ follow.  The claims on the $C^0$ and $C^1$ norm follow as for the transport and Nash errors after using \cref{olowstuff}, applying $\D$ and $\gradperpinverse$, and using Bernstein's inequality in $x$ and $y$ to divide by $\laqone$ due to the presence of the $\Pbarqone$.

We must now calculate the material derivative of the $\Mdot_{O,high}$ portion.  Using that multiplication by $\Lq$ commutes with $\grad \cdot$, $\D$, and $\Dtq$, we can write that
\begin{align*}
    \Dtq &\left( \Lq^2 \D \circ \Pgradbar \left( \grad \cdot \left( O_{high,1} \right) \right)\right)\\
    &= \Lq^2 \left[ \Dtq,\D\circ \Pgradbar \circ (\grad\cdot) \circ \overline{\mathbb{P}}_{\approx \laqone} \right]\Bigg{(} \sum_{k+k'\neq 0} \Xl \Xlprime  \left( \left[ \Pgradk, \aklexp \right](\expk ik)\right)\\
    &\qquad \qquad \qquad \qquad \qquad \qquad \qquad \qquad \qquad \otimes \left( \Pgradkbarperpprime\left(\overline{\wklprime}^\perp\right) \right) \Bigg{)}\\
    &\quad + \Lq^2 \left(\D\circ \Pgradbar \circ (\grad\cdot) \circ \overline{\mathbb{P}}_{\approx \laqone}\right) \Dtq \Bigg{(} \sum_{k+k'\neq 0} \Xl \Xlprime  \left( \left[ \Pgradk, \aklexp \right](\expk ik)\right)\\
    &\qquad \qquad \qquad \qquad \qquad \qquad \qquad \qquad \qquad \otimes \left( \Pgradkbarperpprime\left(\overline{\wklprime}^\perp\right) \right) \Bigg{)}\\
    &=: I + II.
\end{align*}
Since $\left(\D\circ \Pgradbar \circ (\grad\cdot) \circ \overline{\mathbb{P}}_{\approx \laqone}\right)$ is an order zero operator in $x$ and $y$ satisfying the kernel assumptions of the commutator estimate \eqref{commutatorbsv1}, we can write 
\begin{align*}
    \| I \|_{C^0} &\lesssim \|\nabla\Psi_q\|_{C^1} \left\| \left[ \Pgradk, \aklexp \right](\expk ik) \right\|_{C^0} \left\|\Pgradkbarperpprime\left(\wklprime\right) \right\|_{C^0}\\
    &\lesssim \dq^\frac{1}{2} \laq \dqone \laqone^{-\beta}\\
    &\leq \dqtwo \dqone^\frac{1}{2} \laqone.
\end{align*}
Recalling that
$\|\Dtq \Xl\|_{C^0} \leq \muqone$, using parts (2) and (3) of \cref{olowstuff}, and noticing that the singular integral operator $\left(\D\circ \Pgradbar \circ (\grad\cdot) \circ \overline{\mathbb{P}}_{\approx \laqone}\right)$ is bounded on $L^\infty$ due to the frequency localizer and \cref{bernstein}, we can estimate $II$ by 
\begin{align*}
    \|II\|_{C^0} &\lesssim \left\| \Dtq \Xl \right\|_{C^0} \left\| \left[ \Pgradk, \aklexp \right](\expk ik) \right\|_{C^0} \left\|\Pgradkbarperpprime\left(\wklprime\right) \right\|_{C^0}\\
    &\qquad \qquad + \left\| \Dtq \left( \left[ \Pgradk, \aklexp \right](\expk ik) \right) \right\|_{C^0} \left\|\Pgradkbarperpprime\left(\wklprime\right) \right\|_{C^0}\\
    &\qquad \qquad + \left\| \left[ \Pgradk, \aklexp \right](\expk ik) \right\|_{C^0} \left\| \Dtq \Pgradkbarperpprime\left(\wklprime\right) \right\|_{C^0}  \\
    &\lesssim \muqone \dqone \laqone^{-\beta} + \muqone \dqone\\
    &\leq \dqtwo \dqone^\frac{1}{2}\laqone.
\end{align*}
The estimate for the material derivative of $O_{high,2}$ is similar, and we omit it.
\end{proof}

We must now show that the conclusions of \cref{oscillationhigh} hold for the third piece $O_{high,3}$ of the $O_{high}$ error. Before analyzing the $O_{high,3}$ term, we must carefully compute the divergence and determine which pieces of the resulting expression can be absorbed into the error $M_{O,high}$ and which must be absorbed into $\curl\left({Q_{high}}\right)$.  The problematic terms arise when the differential operators fall on $e^{i\laqone k\cdot x}$, since picking up a $\laqone$ makes the resulting term too large to be canceled out by future perturbations.  In the context of the Euler equations, the fact that Beltrami flows are stationary solutions provides an algebraic identity which, when deployed at the right time, shows that the problematic terms can be absorbed into the new pressure.  In our setting, the same principle holds, although its manifestation appears more technical for two reasons. First, the vector field ${Q}$ from \cref{stationarysolutions} is defined as the solution to an elliptic equation via a composition of several differential and integral operators which we must account for. Secondly, we must carefully keep track of the spatial localizer $\Lq$ throughout the decomposition and subsequent estimates. The localizer gives us building blocks which are only stationary solutions to leading order, leaving some extra error terms to estimate. 

\begin{lemma}\label{ohigh3estimates}
The conclusions of \cref{oscillationhigh} hold for $\Lq^2 \grad \cdot O_{high,3}$.
\end{lemma}

\begin{proof}
Calculating the divergence (in $x$ and $y$, i.e. $\grad \cdot$) and setting
$$ f_{klk'l'} = \aklexp \aklexpprime,$$
we have
\begin{align*}
    \Lq^2 \grad \cdot O_{high,3} &= \Lq^2 \grad \cdot \overline{\mathbb{P}}_{\approx \laqone} \left( \sum_{k+k'\neq 0} \Xl \Xlprime f_{klk'l'} \expk ik \otimes \expkprime i\kbarperpprime \right)\\
    &= \Lq^2 \overline{\mathbb{P}}_{\approx \laqone} \sum_{k+k'\neq 0} \Xl \Xlprime \left( \left(\expk ik\right) \otimes \left(\expkprime i\kbarperpprime \right) \right) \cdot \grad\left(f_{klk'l'}\right)\\
    &\quad+ \Lq^2  \overline{\mathbb{P}}_{\approx \laqone} \sum_{k+k'\neq 0} \Xl \Xlprime f_{klk'l'} \grad \cdot \left( \expk ik \otimes \expkprime i\kbarperpprime \right)\\
    &=: \Lq^2 O_{high,3,1} + \Lq^2 O_{high,3,2}.
\end{align*}
The analysis of $O_{high,3,1}$ is simpler due to the fact that the differential operators have landed on $\fkl$.  Estimating its amplitude, we have that
\begin{align*}
    \left\|\Lq^2 O_{high,3,1} \right\|_{C^0} &\lesssim \| \grad\fkl \|_{C^0}\\
    &\lesssim \| \grad\akl \|_{C^0} \| \akl \|_{C^0} + \| \akl \|_{C^0}^2 \left\| \grad e^{i\laqone(\Phi_l-x)\cdot k} \right\|_{C^0}\\
    &\lesssim \dqone \laq + \dqone \laqone^{1-\beta}\\
    &\leq \eta \dqtwo \laqone.
\end{align*}
Recalling that multiplication by $\Lq^2$ commutes with convolution operators and differentiation in $x$ and $y$, we then decompose $\Lq^2 O_{high,3,1}$ using the $\Pgradbar$ and $\Pgradbarperp$ operators into
$$\Lq^2 O_{high,3,1} =\Lq^2  \Pgradbar\left( O_{high,3,1} \right) + \Lq^2 \Pgradbarperp \left( O_{high,3,1} \right). $$
The first term can be plugged into the inverse divergence $\D$ and then absorbed into the error $\M_{O,high}$, while the second term has zero third component and can be absorbed into $\curl(Q_{high})$. The desired $C^0$ and $C^1$ estimates then follow arguing as before. 

We now estimate the material derivative of $\D \circ \Pgradbar \left( \Lq^2 O_{high,3,1} \right)$. We write
\begin{align*}
    \Dtq &\left(\Lq^2 \D \circ \Pgradbar \left( O_{high,3,1} \right) \right)\\
    &= \Lq^2 \Dtq \left( \D \Pgradbar \sum_{k+k'\neq 0} \overline{\mathbb{P}}_{\approx \laqone} \Xl \Xlprime \left( ik \otimes i\kbarperpprime e^{i\laqone(k+k')\cdot x}\right) \grad \fkl \right)\\
    &=\Lq^2 \left[ \Dtq, \D \circ \Pgradbar \circ \overline{\mathbb{P}}_{\approx \laqone} \right] \left( \sum_{k+k'\neq 0} \Xl \Xlprime \left( ik \otimes i\kbarperpprime e^{i\laqone(k+k')\cdot x} \right) \grad \fkl \right)\\
    &\quad + \Lq^2 \D \circ \Pgradbar \circ \overline{\mathbb{P}}_{\approx \laqone} \left( \Dtq \left( \sum_{k+k'\neq 0} \Xl \Xlprime \left( ik \otimes i\kbarperpprime e^{i\laqone(k+k')\cdot x}\right) \grad \fkl \right) \right)\\
    &=: I + II.
\end{align*}
We bound $I$ using \eqref{commutatorbsv1} and the fact that $\D \circ \Pgradbar \circ \overline{\mathbb{P}}_{\approx \laqone}$ is an order $-1$ convolution operator in $x$ and $y$ localized in frequency at $\laqone$, obtaining
\begin{align*}
    \| I \|_{C^0} &\lesssim \| \nabla \Psi_q \|_{C^1} \frac{1}{\laqone} \| \grad \fkl \|_{C^0}\\
    &\lesssim \dq^\frac{1}{2} \laq \frac{1}{\laqone} \dqtwo \laqone\\
    &\leq \dqtwo \dqone^\frac{1}{2} \laqone.
\end{align*}
Before bounding $II$, we write out
\begin{align*}
    \Xl \Xlprime \grad \fkl e^{i\laqone (k+k')\cdot x} &= \Xl \Xlprime \left( \akl \grad \aklprime + \aklprime \grad \akl \right) e^{i\laqone (k+k')\cdot x}\\
    &\qquad + i\laqone \Xl \Xlprime \akl\aklprime \left( \left( D\Phi_l - \Id \right)\kbar + \left( D\Phi_{l'}-\Id \right)\kbar' \right)e^{i\laqone (k+k')\cdot x}.
\end{align*}
Then computing $\Dtq$ of this quantity gives
\begin{align*}
    \Dtq &\left( \grad \fkl e^{i\laqone (k+k')\cdot x} \Xl \Xlprime \right)= \left( \Xl \Xlprime \right)'\left( \akl \grad \aklprime + \aklprime \grad \akl \right) e^{i\laqone (k+k')\cdot x}\\
    &\qquad + i\laqone (\Xl \Xlprime)' \akl \aklprime \left( \left( D\Phi_l - \Id \right)\kbar + \left( D\Phi_{l'}-\Id \right)\kbar' \right)e^{i\laqone (k+k')\cdot x}\\
    &\qquad - \Xl\Xlprime \left( \akl \grad\grad^\perp \Psi_q : \grad \aklprime + \aklprime \grad\grad^\perp\Psi_q : \grad\akl \right)e^{i\laqone (k+k')\cdot x}\\
    &\qquad - i\laqone \Xl \Xlprime \akl \aklprime \left( \grad\grad^\perp\Psi_q: D\Phi_l \cdot k + \grad\grad^\perp\Psi_q: D\Phi_{l'} \cdot k' \right)e^{i\laqone (k+k')\cdot x}.
\end{align*}
Then we can bound $II$ by
\begin{align*}
    \| II \|_{C^0} &\lesssim \frac{1}{\laqone}\bigg{(} \muqone \dqone \laqone + \laqone \muqone \dqone + \dq^\frac{1}{2}\laq \dqone\laq + \laqone \dqone \dq^\frac{1}{2}\laq \bigg{)}\\
    &\leq \dqtwo \dqone^\frac{1}{2}\laqone.
\end{align*}

We now move to the decomposition and estimation of $\Lq^2 O_{high,3,2}$. While in general projecting a vector field onto gradients using $\Pgrad$ induces no gain in regularity, the highest frequency terms in $O_{high,3,2}$ belong to the kernel of the divergence operator.  To see this, let us compute the divergence (now in $x$, $y$, and $z$, i.e. $\nabla\cdot$) of $O_{high,3,2}$:
\begin{align*}
    \nabla \cdot &\left( \overline{\mathbb{P}}_{\approx \laqone} \sum_{k+k'\neq 0} f_{klk'l'} \grad \cdot \left( \expk ik \otimes \expkprime i\kbarperpprime \right) \right)\\
    &= \overline{\mathbb{P}}_{\approx \laqone} \sum_{k+k'\neq 0} f_{klk'l'} \left( i \kbarperpprime \cdot ik \right) (\laqone)^2 e^{i\laqone(k+k')\cdot x} ik \cdot i(k+k')\\
    &\quad + \overline{\mathbb{P}}_{\approx \laqone} \sum_{k+k'\neq 0} i\kbarperpprime \cdot ik \left( \laqone e^{i\laqone(k+k')\cdot x}\right) \nabla\left(f_{klk'l'}\right) \cdot ik \\
    &=: I + II.
\end{align*}
Since the sum is over $k\in \Omega^1$, $k'\in \Omega^2$ where the parity of $l'$ and $l$ matches that of the corresponding sets $\Omega^i$ to which $k$ and $k'$ belong, the coefficients $\fkl$ allow for the application of the algebraic identity \eqref{algebraicidentity} from \cref{stationarysolutions}. Therefore, $I$ is equal to zero pointwise in $\torus$, showing that the problematic terms are annihilated by the divergence. Then we can write that
\begin{align*}
    \nabla \mathcal{F} &:= \Pgrad \left( O_{high,3,2} \right)\\
    &= \nabla \circ (-\Delta)^{-1} \circ (\nabla \cdot) \left( O_{high,3,2} \right)\\
    &= \nabla \circ (-\Delta)^{-1} (II).
\end{align*}
Although the third component of the frequency support of $\mathcal{F}$ is not compact, the first two components are supported in an annulus centered around $\laqone$, and so Bernstein's inequality gives that
\begin{align*}
    \| \nabla \mathcal{F} \|_{C^0} &\lesssim \frac{1}{\laqone} \| II \|_{C^0}\\
    &\lesssim \frac{1}{\laqone} \laqone \| \fkl \|_{C^1}\\
    &\leq \eta \dqtwo \laqone.
\end{align*}
Conversely, after setting $\mathcal{G}:= \Pcurl (O_{high,3,2})$, we have that
\begin{align*}
    \left\| \curl (\mathcal{G}) \right\|_{C^0} &= \left\| \Pcurl (O_{high,3,2}) \right\|_{C^0}\\
    &\leq \left\| O_{high,3,2} \right\|_{C^0} + \left\| \nabla \mathcal{F} \right\|_{C^0}\\
    &\lesssim  \laqone \| \fkl \|_{C^0} + \left\| \nabla \mathcal{F} \right\|_{C^0}\\
    &\lesssim \dqone \laqone.
\end{align*}
Furthermore, since $\mathcal{G}=(-\Delta)^{-1} \circ \curl (O_{high,3,2})$ is given by an operator of order $-1$ applied to $O_{high,3,2}$, by the presence of $\overline{\mathbb{P}}_{\approx \laqone}$ and Bernstein's inequality we see that $\| \mathcal{G} \|_{C^0} \lesssim \dqone$.

We are now ready to decompose $\Lq^2 O_{high,3,2}$.  
\begin{align*}
    \Lq^2 O_{high,3,2} &= \Lq^2 \Pgrad\left(O_{high,3,2}\right) + \Lq^2 \Pcurl(O_{high,3,2})\\
    &= \Lq^2 \nabla \mathcal{F} + \Lq^2 \curl( \mathcal{G} ) \\
    &= \begin{bmatrix}
           \partial_x\left(\Lq^2 \mathcal{F}\right)\\
           \partial_y\left(\Lq^2 \mathcal{F}\right)\\
           \Lq^2 \partial_z \mathcal{F}
         \end{bmatrix} + 
    \begin{bmatrix}
           -\mathcal{G}^2 \partial_z(\Lq^2)\\
           \mathcal{G}^1 \partial_z(\Lq^2)\\
           0
         \end{bmatrix} + \curl\left(  \Lq^2 \mathcal{G} \right).
\end{align*}
The first term can now be absorbed into the error $M_{O,high}$ after applying $\D$, while the third term can be absorbed into $\curl(Q_{high})$. The estimates on the amplitudes, $C^1$ norms, and spatial supports follow from the above estimates on $\mathcal{F}$ and $\mathcal{G}$. Before addressing the second term, which we shall denote
\begin{align*}
    \mathcal{L}:=\begin{bmatrix}
           -\mathcal{G}^2 \partial_z(\Lq^2)\\
           \mathcal{G}^1 \partial_z(\Lq^2)\\
           0
         \end{bmatrix}
\end{align*}
let us calculate the material derivative of the first.
\begin{align*}
    \Dtq &\left( \D \left( \Lq^2 \nabla \mathcal{F} \right) \right) = \Lq^2 \Dtq \left( \D (\nabla \mathcal{F}) \right)\\
    &= \Lq^2 \left[ \Dtq, \D \right](\nabla \mathcal{F}) + \Lq^2 \D \left( \Dtq (\nabla \mathcal{F}) \right).
\end{align*}
We can bound the first term using \eqref{commutatorbsv1} and the fact that $\nabla\mathcal{F}$ is supported in an annulus of radius $\laqone$ in $x$ and $y$ frequencies by 
\begin{align*}
\left\| \Lq^2 \left[ \Dtq, \D \right](\nabla \mathcal{F}) \right\|_{C^0} &\lesssim \frac{1}{\laqone} \| \nabla \Psi_q \|_{C^1} \| \nabla \mathcal{F} \|_{C^0} \\
&\lesssim \dq^\frac{1}{2}\laq \frac{1}{\laqone} \dqtwo \laqone\\
&\leq \dqtwo \dqone^\frac{1}{2} \laqone.
\end{align*}
We decompose the second term further as
\begin{align*}
    \Lq^2 \D \left( \Dtq (\nabla \mathcal{F} ) \right) &= \Lq^2 \D \left( \left[ \Dtq, \Pgrad \right](O_{high,3,2})\right) + \Lq^2 \D \left( \Pgrad \left( \Dtq (O_{high,3,2}) \right) \right).
\end{align*}
Using the fact that $O_{high,3,2}$ is supported in an annulus of size $\laqone$ in $x$ and $y$ frequencies, we can bound the first term using the commutator estimate from \cref{commutatorconstantin} by
\begin{align*}
    \left\| \Lq^2 \D \left( \left[ \Dtq, \Pgrad \right](O_{high,3,2})\right) \right\|_{C^0} &\lesssim \frac{1}{\laqone} \| \nabla \Psi_q \|_{C^{1+\alpha}} \| O_{high,3,2} \|_{C^\alpha}\\
    &\lesssim \frac{1}{\laqone} \dq^\frac{1}{2} \laq^{1+\alpha}\dqone \laqone^{1+\alpha}\\
    &\leq \dqtwo \dqone^\frac{1}{2} \laqone
\end{align*}
if $\alpha$ is small enough. Then for the second term, we can write
\begin{align*}
    &\left\| \Lq^2 \D \left( \Pgrad \left( \Dtq (O_{high,3,2}) \right) \right) \right\|_{C^0} \\
    &\qquad = \left\| \Lq^2 \D \circ \Pgrad \left( \left[ \Dtq, \overline{\mathbb{P}}_{\approx \laqone} \right] \left( \sum_{k+k'\neq 0} \akl \aklprime e^{i\laqone \Phi_l \cdot k} e^{i\laqone \Phi_l \cdot k'} \laqone k\otimes \kbarperpprime (k+k') \right) \right) \right\|_{C^0}\\
    &\qquad \lesssim \frac{1}{\laqone} \| \nabla \Psi_q \|_{C^1} \laqone \| \akl \|_{C^0}^2\\
    &\qquad \leq \frac{1}{\laqone}\dq^\frac{1}{2}\laq \laqone \dqone \\
    &\qquad \leq \dqtwo \dqone^\frac{1}{2} \laqone.
    \end{align*}

We now return to $\mathcal{L}$. Since $\mathcal{L}$ has derivatives on $\Lq$ rather than $\mathcal{G}$, it is significantly smoother than $\curl(\mathcal{G})$. We decompose $\mathcal{L}$ as
$$ \mathcal{L} = \Pgradbar \left( \mathcal{L} \right) + \Pgradbarperp \left( \mathcal{L} \right). $$
Then $\Pgradbar \left( \mathcal{L} \right)$ is absorbed into the error $M_{O,high}$ after applying $\D$, while $\Pgradbarperp \left( \mathcal{L} \right)$ can be absorbed into the curl since it has zero third component. Estimating the amplitude of $\mathcal{L}$, we can write 
$$ \| \mathcal{L} \|_{C^0} \lesssim \| \partial_z \Lq \|_{C^0} \| \mathcal{G} \|_{C^0} \leq l_{q+1} \dqone \leq \eta \dqtwo \laqone, $$ 
and thus the desired $C^0$ and $C^1$ estimates follow from Bernstein's inequality and the fact that $\mathcal{L}$ is compactly supported in frequency in $x$ and $y$.

Finally, it remains to estimate the material derivative of $\D\Pgradbar \mathcal{L}$.
\begin{align*}
    &\left\| \partial_z(\Lq^2) \Dtq \left( \D \circ \Pgradbar \circ (-\Delta)^{-1} \circ \curl \circ \overline{\mathbb{P}}_{\approx \laqone} \left( \sum_{k+k'\neq 0} \akl \aklprime e^{i\laqone \left(\Phi_l \cdot k+ \Phi_l \cdot k'\right)} k\otimes \kbarperpprime (k) \right) \right) \right\|_{C^0}\\
    &\leq l_{q+1} \left\| \left[ \Dtq, \D \circ \Pgradbar \circ (-\Delta)^{-1} \circ \curl \circ \overline{\mathbb{P}}_{\approx \laqone} \right] \left( \sum_{k+k'\neq 0} \akl \aklprime e^{i\laqone \left(\Phi_l \cdot k+ \Phi_l \cdot k'\right)} k\otimes \kbarperpprime (k) \right) \right\|_{C^0}\\
    &\qquad \qquad \lesssim l_{q+1} \frac{1}{\laqone^2} \| \nabla \Psi_q \|_{C^{1+\alpha}} \| O_{high,3,2} \|_{C^\alpha}\\
    &\qquad \qquad \lesssim  l_{q+1} \frac{1}{\laqone^2} \dq^\frac{1}{2} \laq^{1+\alpha} \dqone \laqone^{1+\alpha}\\
    &\qquad \qquad \leq \dqtwo \dqone^\frac{1}{2}\laqone.
\end{align*}
We remark that estimating the commutator of $\Dtq$ and $\D \circ \Pgradbar \circ (-\Delta)^{-1} \circ \curl \circ \Pbarqone$ can be done following the ideas of the proof of \eqref{commutatorbsv1} if one is willing to pay a $C^\alpha$ norm on $\grad^2\Psi_q$ and $O_{high,3,2}$, which is acceptable considering that $l_{q+1}$ is much smaller than $\laqone$.
\end{proof}

\subsubsection{Estimates for \texorpdfstring{O\textsubscript{low}}{O low}}

$O_{low}$ is given by
$$O_{low} = \Lq^2 \grad \cdot \left( \sum_{k+k'=0} \Xl\Xlprime \Pgradk\left(\wkl\right) \otimes  \Pgradkbarperpprime\left(\overline{\wklprime}^\perp \right) \right) + \Lq^2 \grad \cdot \Mdot_q. $$
Recall that the choice of vectors $k$ implies that if $k=-k'$, then $l$ and $l'$ have the same parity. For $l$ and $l'$ with the same parity, $\sum_{l'}\Xl\Xlprime = \Xl^2$. In order to isolate the terms which cancel out $\grad\cdot\Mdot_q$, we rewrite $O_{low}$ as 
\begin{align*}
    O_{low} &= \Lq^2 \grad \cdot \Bigg{(} \sum_{k+k'=0} \Xl^2 \left( \left[ \Pgradk, \akl e^{i\laqone(\Phi_l-x)\cdot k} \right](e^{i\laqone k\cdot x}ik) + \wkl \right) \\
    &\qquad \qquad \qquad \otimes \left( \left[ \Pgradkbarperpprime, \aklprime e^{i\laqone(\Phi_{l'}-x)\cdot k'} \right](e^{i\laqone k'\cdot x}ik') + \overline{\wklprime}^\perp \right)\Bigg{)}\\
    &\qquad + \Lq^2 \grad \cdot \Mdot_q \\
    &=: O_{low,1} + O_{low,2} + O_{low,3} + O_{low,4} + O_{low,5}
\end{align*}
where 
\begin{align*}
O_{low,1} := \Lq^2 \grad \cdot \Bigg{(} \left( \sum_{k+k'=0} \Xl^2\left[ \Pgradk,\aklexp \right](\expk ik) \right) \otimes \Pgradkbarperp(\overline{\wklprime}^\perp) \Bigg{)},
\end{align*}
$$ O_{low,2} := \Lq^2 \grad \cdot \Bigg{(}  \sum_{k+k'=0} \Xl^2 \wkl   \otimes \left( \left[ \Pgradkbarperpprime, \aklexpprime \right](\expkprime i\kbar'^\perp) \right) \Bigg{)},$$
$$ O_{low,3} = \Lq^2 \grad \cdot \left( \Mdot_q - \Mdot_{q,\ell} \right) $$
$$ O_{low,4} := \Lq^2 \grad \cdot \Bigg{(} \sum_{k+k'=0} \Xl^2 \left( \wkl \otimes \overline{\wklprime}^\perp - M_{q,l} \right) \Bigg{)}$$
$$ O_{low,5} := \Lq^2 \grad \cdot \left( \sum_l \Xl^2 \left( \M_{q,\ell} - \M_{q,l} \right) \right). $$
We see that by construction, 
$$ O_{low,4} = \Lq^2 \grad \cdot \left( \frac{1}{2}\sum_{k} \Xl^2 \left( |\akl|^2 k \otimes \kbar^\perp - M_{q,l} \right) \right) = 0,$$
giving us the required cancellation. Thus, it remains to decompose and estimate $O_{low,1}$, $O_{low,2}$, and $O_{low,3}$, and $O_{low,5}$. We state the results as follows.
\begin{lemma}\label{oscillationlow}
The low frequency portion of the oscillation error $O_{low}$
is equal to
$$ \curl\left({Q}_{low}\right) + \grad \cdot \Mdot_{O,low} $$
with the estimates
$$ \| {Q}_{low} \|_{C^0} \leq \dqone, \qquad \| {Q}_{low} \|_{C^1} \leq \dqone \laqone $$
$$ \| \Mdot_{O,low} \|_{C^0} \leq \eta\dqtwo, \qquad \| \Mdot_{O,low} \|_{C^1} \leq \dqtwo \laqone, \qquad \| \Dtq \Mdot_{O,low} \|_{C^0} \leq \dqtwo \dqone^\frac{1}{2} \laqone. $$
Furthermore, $Q_{low}$ and $\Mdot_{O,low}$ are supported in the set $$\mathbb{T}^2 \times \left[ \frac{1}{l_{q+1}}, 2\pi - \frac{1}{l_{q+1}} \right].$$
\end{lemma}
\begin{proof}
We start by decomposing $O_{low,1}$ as 
$$ O_{low,1} = \Pgradbar \left( O_{low,1} \right) + \Pgradbarperp \left( O_{low,1} \right). $$
As $\Pgradbar$ and $\Pgradbarperp$ are convolution operators in $x$ and $y$ only, they commute with multiplication by $\Lq^2$, and the claim on the spatial supports follows. The first term is absorbed into $M_{O,low}$ after applying $\D$, while the second term is absorbed into $\curl( Q_{low})$ by inverting $\grad^\perp$. We estimate the $\Pgradbar$ portion now.
\begin{align*}
    \|\D \Pgradbar &O_{low,1}\|_{C^0}\\ 
    &\leq \sup_{k+k'=0} \bigg{\|} \D\Pgradbar \grad \cdot \bigg{(}  \left[ \Pgradk, \aklexp \right](\expk k) \otimes \Pgradkbarperpprime(\overline{\wklprime}^\perp)\bigg{)}\bigg{\|}_{C^0(\supp \Xlprime)}\\
    &\leq \left\| \left[ \Pgradk, \aklexp \right](\expk k) \right\|_{C^0} \left\| \Pgradkbarperpprime\left( \overline{\wklprime}^\perp \right) \right\|_{C^\alpha} \\
    &\qquad + \left\| \left[ \Pgradk, \aklexp \right](\expk k) \right\|_{C^\alpha} \left\| \Pgradkbarperpprime\left( \overline{\wklprime}^\perp \right) \right\|_{C^0} \\
    &\lesssim \dqone \laqone^{\alpha-\beta}\\
    &\leq \eta \dqtwo
\end{align*}
after using \cref{help} and assuming $\alpha$ is sufficiently small. The estimate for the $\Pgradbarperp$ portion follows by simply replacing $\D \circ \Pgradbar$ with $(-\lap)^{-1} \circ (\grad^\perp \cdot)$ in the above argument.

To calculate the $C^1$ norms, we write
\begin{align*}
    &\left\| \nabla \D \Pgradbar O_{low,1} \right\|_{C^0} \lesssim \left\| \partial_z (\Lq^2) \right\|_{C^0}  \\
    &\qquad\quad \times \sup_{k+k'=0} \bigg{\|} \D\Pgradbar \grad \cdot \bigg{(}  \left[ \Pgradk, \aklexp \right](\expk k) \otimes \Pgradkbarperpprime(\overline{\wklprime}^\perp)\bigg{)}\bigg{\|}_{C^0(\supp \Xlprime)} \\
    &\qquad + \sup_{k+k'=0} \bigg{\|} \D\Pgradbar \grad \cdot \bigg{(}  \left[ \Pgradk, \aklexp \right](\expk k) \otimes \Pgradkbarperpprime(\overline{\wklprime}^\perp)\bigg{)}\bigg{\|}_{C^1(\supp \Xlprime)}\\
    &\quad \leq \eta \dqtwo \laqone
\end{align*}
after arguing as above. The decomposition and estimate for $O_{low,2}$ is analogous, and we omit the calculation.

Note that $O_{low,3} = \Lq^2 \grad \cdot\left( \Mdot_q - \Mdot_{q,\ell}\right)$ is already the divergence of a suitable matrix. To estimate the $C^0$ norm, standard mollification estimates give
\begin{align*}
    \left\| \Mdot_q - \Mdot_{q,\ell} \right\|_{C^0} &\leq \dqone \laq {\laq^{-\frac{3}{4}}\laqone^{-\frac{1}{4}}}\\
    &\leq \eta\dqtwo.
\end{align*}
The $C^1$ norm is then easily controlled by $2 \left\| \Mdot_q \right\|_{C^1} = 2\dqone \laq \leq \dqtwo \laqone$, showing the desired result.  

For $O_{low,5}$ we recall that for $t=\frac{l}{\muqone}$,
$$\Mdot_{q,\ell}(t) = \Mdot_{q,l}(t)$$ for all $x \in \torus$, and that $$\Dtq(\Mdot_{q,\ell} - \Mdot_{q,l}) = \Dtq \Mdot_{q,\ell}.$$ 
Before calculating the $C^0$ and $C^1$ norm, let us calculate the material derivative of $\M_{q,\ell}$.
\begin{align*}
    \Dtq \M_{q,\ell} &= \left( \Dtq \M_q \right) \ast \phi_q + \grad^\perp\Psi_q \cdot \grad \M_{q,\ell} - \left( \grad^\perp\Psi_q \cdot \grad \M_q \right) \ast \phi_q
\end{align*}
A simple calculation shows that the commutator 
\begin{align*}
    \left\| \left[ \grad^\perp\Psi_q \cdot \grad, \phi_q \ast \right]( \M_{q} ) \right\|_{C^0} &\leq \| \grad^\perp\Psi_q \|_{C^1} \| \M_q \|_{C^1} \ell,
\end{align*}
thus showing that 
$$\left\| \Dtq \M_{q,\ell} \right\|_{C^0} \leq \dqone \laq \dq^\frac{1}{2} \laq \ell \leq \dqtwo \dqone^\frac{1}{2} \laqone. $$
In addition, we obtain that 
\begin{align*}
    \left\| \Dtq \M_{q,\ell} \right\|_{C^1} \leq \dqone\dq^\frac{1}{2}\laq^2.
\end{align*}
Applying the transport estimate from \cref{transport} and the inductive assumption \eqref{inductivetransport}, we find
\begin{align*}
 \left\| \Lq^2 \sum_l \Xl^2 \left( \Mdot_{q,\ell} - \Mdot_{q,l} \right) \right\|_{C^0} & \leq \sup_{l} \left\|\Mdot_{q,\ell} - \Mdot_{q,l}\right\|_{C^0(\supp \Xl)}\\
 &\leq \frac{1}{\muqone} \dqone \dq^\frac{1}{2} \laq \\
 &\leq \eta\dqtwo.
\end{align*}
Applying the transport estimate \cref{transport} then shows that
\begin{align*}
    \left\| \M_{q,\ell} - \M_{q,l} \right\|_{C^1} &\leq \frac{1}{\muqone} \dqone \dq^\frac{1}{2} \laq^2 \\
    &=\dqone^\frac{3}{4} \dq^\frac{1}{4} \laq^\frac{3}{2} \laqone^{-\frac{1}{2}}\\
    &\leq \dqtwo \laqone,
\end{align*}
providing the desired $C^1$ bound after recalling that $\partial_z \Lq$ is small.

Moving now to the material derivative, we have that 
\begin{align*}
    \Dtq M_{O,low}&= \Dtq \D \Pgradbar \grad \cdot \Bigg{(} \Lq^2 \sum_{k+k'=0} \Xl^2 \left[ \Pgradk, \aklexp \right](\expk ik) \otimes \Pgradkbarperpprime(\overline{\wklprime}^\perp)  \\
    & \qquad\qquad\qquad\qquad + \Lq^2 \sum_{k+k'=0} \Xl^2 \wkl \otimes \left[ \Pgradkbarperpprime, \aklexpprime \right](\expkprime i\kbar'^\perp) \Bigg{)}\\
    &\qquad\qquad + \Dtq \left( \M_q - \M_{q,\ell} +  \Lq^2 \sum_l \Xl^2\left( \Mdot_{q,\ell} - \Mdot_{q,l} \right) \right) \\
    &=: \Dtq \D \Pgradbar \grad \cdot \Omega + \Dtq \left( \M_q - \M_{q,\ell} +  \Lq^2 \sum_l \Xl^2\left( \Mdot_{q,\ell} - \Mdot_{q,l} \right) \right).
\end{align*}
The second and third terms are the easiest to analyze and we dispense with it first. Since $\Dtq \Lq^2 = \Dtq \Mdot_{q,l} = 0$, we can write that 
\begin{align*}
    \left\| \Dtq \left( \Lq^2 \sum_l \Xl^2\left( \Mdot_{q,\ell} - \Mdot_{q,l} \right) \right) \right\|_{C^0} &\leq \left\| \Dtq \Xl^2 \right\|_{C^0} \left\| \Mdot_q - \Mdot_{q,l} \right\|_{C^0} + \left\| \Dtq \Mdot_{q,\ell} \right\|_{C^0}\\
    &\leq \dqtwo \dqone^\frac{1}{2}\laqone
\end{align*}
after applying the previous estimate on $\Dtq \M_{q,\ell}$. In addition, we have that 
\begin{align*}
    \left\| \Dtq \left( \M_q -\M_{q,\ell} \right) \right\|_{C^0} \leq \dqtwo \dqone^\frac{1}{2}\laqone
\end{align*}
after applying the inductive assumption and the estimate on $\Dtq \M_{q,\ell}$.

The first step towards estimating the other term is to estimate the commutator of $\Dtq$ and $\D \Pgradbar \grad \cdot$ applied to $\Omega$ using \cref{commutatorconstantin}. We can write
\begin{align*}
    &\left\| \left[ \Dtq, \D\Pgradbar\grad\cdot \right] \left( \Omega \right) \right\|_{C^0} \leq \|\nabla\Psi_q\|_{{C}^{1+\alpha}} \| \Omega \|_{{C}^\alpha}\\
    &\qquad \qquad \leq \| \nabla\Psi_q \|_{{C}^{1+\alpha}} \Bigg{(} \left\| \left[ \Pgradk, \aklexp \right](\expk ik) \right\|_{{C}^\alpha} \left\| \Pgradkbarperpprime(\overline{\wklprime}^\perp) \right\|_{C^0}\\
    &\qquad\qquad\qquad + \left\| \left[ \Pgradk, \aklexp \right](\expk ik) \right\|_{C^0} \left\| \Pgradkbarperpprime(\overline{\wklprime}^\perp) \right\|_{{C}^\alpha}\\
    &\qquad\qquad\qquad + \left\| \left[ \Pgradkbarperpprime, \aklexp \right](\expk ik') \right\|_{{C}^\alpha} \left\| \wkl \right\|_{C^0}\\
    &\qquad\qquad\qquad + \left\| \left[ \Pgradkbarperpprime, \aklexp \right](\expk ik') \right\|_{C^0} \left\| \wkl \right\|_{{C}^\alpha} \Bigg{)}\\
    &\qquad \qquad \lesssim \dq^\frac{1}{2}\laq^{1+\alpha} \dqone \laqone^\alpha\\
    &\qquad \qquad \leq \dqtwo \dqone^\frac{1}{2} \laqone
\end{align*}
if $\alpha$ is small enough. Therefore, it remains to estimate 
\begin{align*}
\D \Pgradbar &\grad \cdot \Dtq \left(\Lq^2 \sum_{k+k'=0} \Xl^2 \left[ \Pgradk, \aklexp \right](\expk ik) \otimes \Pgradkbarperpprime(\overline{\wklprime}^\perp) \right) \\
& \qquad + \D \Pgradbar \grad \cdot \Dtq \left(\Lq^2 \sum_{k+k'=0} \Xl^2 \wkl \otimes \left[ \Pgradkbarperpprime, \aklexpprime \right](\expkprime i\kbar'^\perp) \right).
\end{align*}
We first simplify the above expression by noticing that 
\begin{align*}
\Dtq &\Bigg{(} \Lq^2 \sum_{k+k'=0} \Xl^2 \left[ \Pgradk, \aklexp \right](\expk ik) \otimes \Pgradkbarperpprime(\overline{\wklprime}^\perp) \\
& \qquad + \Lq^2 \sum_{k+k'=0} \Xl^2 \wkl \otimes \left[ \Pgradkbarperpprime, \aklexpprime \right](\expkprime i\kbar'^\perp) \Bigg{)} \\
&= \Dtq \left( \Lq^2 \sum_{k+k'=0} \Xl^2 \left( \Pgradk \left( \wkl \right) \otimes \Pgradkbarperpprime \left( \overline{\wklprime}^\perp \right) - \wkl \otimes \overline{\wklprime}^\perp \right) \right)\\
&= \Lq^2 \sum_{k+k'=0} \Dtq \left(\Xl^2\right) \left( \Pgradk \left( \wkl \right) \otimes \Pgradkbarperpprime \left( \overline{\wklprime}^\perp \right) - \wkl \otimes \overline{\wklprime}^\perp \right)\\
&\qquad + \Lq^2 \sum_{k+k'=0} \Xl^2 \left( \Dtq \left( \Pgradk \left( \wkl \right) \otimes \Pgradkbarperpprime \left( \overline{\wklprime}^\perp \right) \right) \right).
\end{align*}
Notice that the terms with the projection operators $\Pgradk$ and $\Pgradkbarperpprime$ are supported in an annulus in $x$ and $y$ frequencies, and so the singular integral operator $\D \Pgradbar \grad \cdot$ is bounded on $L^\infty$ for these terms by \cref{bernstein}. Then the entire expression is bounded by
\begin{align*}
    &\left\| \Dtq \left( \Xl^2 \right) \right\|_{C^0} \left( \left\| \Pgradk \left( \wkl \right) \otimes \Pgradkbarperpprime\left( \overline{\wklprime}^\perp\right) \right\|_{C^0} + \left\| \wkl \otimes \wklprime \right\|_{C^\alpha} \right)\\
    &\qquad + \left\| \Dtq \left(  \Pgradk \left( \wkl \right) \otimes \Pgradkbarperpprime\left( \overline{\wklprime}^\perp\right) \right) \right\|_{C^0} \\
    &\leq \muqone \left( \dqone + \dqone\laqone^\alpha \right) + \muqone \dqone\\
    &\leq \dqtwo \dqone^\frac{1}{2}\laqone,
\end{align*}
finishing the proof.
\end{proof}

\section{Energy Increment}\label{energyincrement}

In this section, we show that the inductive assumptions \eqref{inductiveenergyprofileone} and \eqref{inductiveenergyprofiletwo} hold with $q$ replaced by $q+1$.  The proof follows estimates of the Hamiltonian increment from \cite{BCV2018} and is thus split up into a preliminary lemma and subsequent proposition.

\begin{lemma}\label{preliminaryenergylemma}
If $t \in \supp{\Xl}$, then 
\begin{align*}
    \bigg{|} \int_{\torus} \left( |\nabla\Psi_q(t)|^2 - |\nabla\Psi_q \left(\frac{l}{\muqone}\right)|^2 \right)\bigg{|} + \bigg{|} e(t) - e\left(\frac{l}{\muqone}\right) \bigg{|} \leq \frac{\dqtwo}{16}.
\end{align*}
Furthermore, for $\rho_l \neq 0$
\begin{align*}
    | \rho(t) -\rho_l | \leq \frac{\dqtwo}{16}
\end{align*}
and 
\begin{align*}
 e(t) - \int_{\torus} |\nabla\Psi_q(t)|^2 \geq \frac{7\dqtwo}{16}. 
\end{align*}
If $\rho_l = 0$, then 
\begin{align*}
    e\left(\frac{l}{\muqone}\right) - \int_{\torus} |\nabla\Psi_q(t)|^2 \leq \frac{9\dqtwo}{16} \quad \textnormal{and} \quad \Mdot_q(\cdot,t)\equiv 0.
\end{align*}
\end{lemma}

\begin{proof}
Using that $\nabla\Psi_q$ solves \eqref{inductiveequation} and multiplying by $\nabla\Psi_q$ and integrating by parts, we obtain
\begin{align*}
    \left| \int_{\torus} \left( |\nabla\Psi_q(t)|^2 -\left|\nabla\Psi_q\left(\frac{l}{\muqone}\right)\right|^2 \right) \right| &= \left| \int_{\torus} \int_{\frac{l}{\muqone}}^t  \Mdot_q : \grad \nabla\Psi_q \right|\\
    &\leq \left(t-\frac{l}{\muqone}\right) \dqone\dq^\frac{1}{2}\laq\\
    &\lesssim \frac{4}{\muqone}\dqone\dq^\frac{1}{2}\laq\\
    &\leq \frac{\dqtwo}{32}.
\end{align*}
The bound
$$ \left| e(t) - e\left(\frac{l}{\muqone}\right) \right| \lesssim \frac{1}{\muqone} \leq \frac{\dqtwo}{32} $$
follows from the smoothness of $e(t)$. Summing both estimates, the first claim is shown.  The second claim follows from the first and the definition of $\rho(t)$. The final bound follows from the definition of $\rho(t)$, the first bound, and \eqref{inductiveenergyprofiletwo}.  
\end{proof}

\begin{prop}\label{energyincrementlemma}
If $\rho_l \neq 0$ and $t\in\supp \Xl$, then \begin{align*}
    \frac{\dqtwo}{4} \leq e(t) - \int_{\torus} |\nabla\Psi_{q+1}(t)|^2 \leq \frac{3\dqtwo}{4}.
\end{align*}
If not, however, then 
\begin{align*}
    e(t) - \int_{\torus} |\nabla\Psi_{q+1}(t)|^2 \leq \frac{9\dqtwo}{16} \quad \textnormal{and} \quad \Mdot_{q+1}(\cdot,t)\equiv 0.
\end{align*}
\end{prop}

\begin{proof}
Beginning with the case when $\rho_l = 0$ and $t\in\supp \Xl$, we have that $\nabla\W(t) = 0$, which implies that $\Mdot_{q+1}(t) = \Mdot_q(t) = 0$ and
$$ e(t) - \int_{\torus} |\nabla\Psi_{q+1}(t)|^2 = e(t) - \int_{\torus} |\nabla\Psi_{q}(t)|^2 \leq \frac{9\dqtwo}{16}. $$

Moving to the case when $\rho_l \neq 0$ and $t \in \supp \Xl$, then by the frequency and spatial support of $\nabla\Psi_q$ and $\nabla\W$, we have that
\begin{align*}
e(t) - \int_{\torus} |\nabla\Psi_{q+1}(t)|^2 &= e(t) - \int_{\torus} |\nabla\Psi_{q}(t)|^2\\
&\qquad - \int_{\torus} |\nabla(\Lq\W)(t)|^2 - 2\int_{\torus} \nabla\Psi_q(t) \cdot \nabla (\Lq \W)(t)\\
&= e(t) - \int_{\torus} |\nabla\Psi_{q+1}(t)|^2\\
&\qquad - \int_{\torus} |\nabla(\Lq\W)(t)|^2 - 2\int_{\torus} \nabla \Psi_q(t) \cdot \nabla \W(t).
\end{align*}
We have to estimate
\begin{align*}
    \int_{\torus} \left|\nabla(\Lq\W)\right(t)|^2 + 2\int_{\torus} \nabla\Psi_q(t)\cdot \nabla \W (t) =: I + II.
\end{align*}
Using \eqref{inductivefrequencysupport} and the definition of $\Pgradk$ to see that $\nabla\Psi_q$ and $\nabla\W$ are supported in disjoint sets in frequency, we see that $II=0$. Writing out $I$ gives
\begin{align*}
    \int_{\torus} \left| \nabla (\Lq \W)(t) \right|^2 &= \int_{\torus} \Lq^2 \nabla \W(t) \cdot \nabla \W(t)\\
    &\qquad + 2 \int_{\torus} \Lq\partial_z \Lq \W \partial_z \W + \int_{\torus} \left( \partial_z \Lq \right)^2 \left( \W \right)^2\\
    &= I_1 + I_2 + I_3.
\end{align*}
We can control $I_2$ by 
$$ \left| 2 \int_{\torus}  \Lq\partial_z \Lq \W \partial_z \W \right| \leq l_{q+1}\frac{\dqone}{\laqone} $$
and 
$I_3$ by
$$ \left| \int_{\torus} \left( \partial_z \Lq \right)^2 \left( \W \right)^2 \right| \leq l_{q+1}^2 \frac{\dqone}{\laqone^2}. $$
Writing out $I_1$ gives
\begin{align*}    
I_1 &=  \sum_{kl} \int_{\torus} \Lq^2 \Pgradk(\Xl\wkl)\cdot\Pgradminusk(\Xl w_{-kl})\\
&=  \sum_{kl} \int_{\torus} \Lq^2 \Xl^2 \bigg{[} \wkl \cdot w_{-kl} + \left[ \Pgradk, \aklexp \right](\expk ik) \cdot w_{-kl}\\
&\qquad\qquad + w_{kl}\cdot \left[ \Pgradminusk, \aminusklexp\right](e^{-i\laqone k}ik)\\
&\qquad\qquad + \left[ \Pgradk, \aklexp \right](\expk ik) \cdot \left[ \Pgradminusk, \aminusklexp\right](e^{-i\laqone k}ik) \bigg{]}\\
&= \sum_{kl} \int_{\torus} \Lq^2 \Xl^2 |a_{kl}|^2 + O\left( \dqone\laqone^{-\beta} \right)\\
&= \sum_l \Xl^2 \rho_l \int_{\torus} \Lq^2 + O\left( \dqone\laqone^{-\beta} \right).
\end{align*}
after applying the commutator estimate \eqref{commutatorbsv1} and \eqref{aklidentity}. Then applying the definition of $\rho_l$ given in \eqref{definitionofrho} finishes the proof.
\end{proof}

\section{Proof of Main Results}\label{Proofofmainresults}

\begin{proof}[Proof of \cref{inductiveproposition}]
We show that each inductive step holds with $q$ replaced by $q+1$. Referring to the statements of \cref{transporterror}, \cref{nasherror}, \cref{oscillationhigh}, and \cref{oscillationlow}, we have that $\nabla\Psi_{q+1}$ solves
$$ \partial_t \nabla\Psi_{q+1} + \grad^\perp \Psi_{q+1} \cdot \grad \nabla\Psi_{q+1} =\curl(Q_{q+1}) + \grad \cdot \Mdot_{q+1} $$
where 
$$Q_{q+1}=Q_T+Q_N+Q_{high}+Q_{low}, \qquad \Mdot_{q+1}=\Mdot_T+\Mdot_N+\Mdot_{high}+\Mdot_{low}$$
and thus \eqref{inductiveequation} is satisfied. The inductive step \eqref{inductivefrequencysupport} follows from the frequency support of $\W\Lq$. \eqref{inductivespatialsupport}-\eqref{inductivecurl} follow directly from the statements of \cref{help}, \cref{transporterror}, \cref{nasherror}, \cref{oscillationhigh}, and \cref{oscillationlow}. Finally, \eqref{inductiveenergyprofileone} and \eqref{inductiveenergyprofiletwo} follow from \cref{energyincrementlemma}.
\end{proof}

\begin{proof}[Proof of \cref{eulerinductiveproposition}]
Towards the purpose of constructing solutions to 2D Euler, one first eliminates the inductive assumption \eqref{inductivespatialsupport} on the spatial support and defines $\Lq\equiv 1$ for all $q$. Next, choose the first set of frequency modes to have zero third component. Then it is easy to see that $\Mdot_1$ is of the specified block form. Continuing to apply \cref{frequencymodes} by choosing modes with zero third component since the third row of $\Mdot_q$ is empty gives immediately that $\partial_z(\Psi_{q+1}-\Psi_q)\equiv 0$, and therefore $\Psi$ depends only on $x$, $y$, and $t$.
\end{proof}

\begin{proof}[Proof of \cref{maintheorem}]
From the estimate $\| \wkl  \|_{C^1} + \| \Lq \|_{C^1} \leq \dqone^\frac{1}{2}\laqone$, we have that
\begin{align*}
    \left\| \nabla \left( \nabla\Psi_{q+1} - \nabla\Psi_q \right) \right\|_{C^0} &= \left\| \nabla^2 \left(\Lq\W)\right) \right\|_{C^0}\\
    &\leq \dqone^\frac{1}{2} \laqone.
\end{align*}
We claim that the time derivative $\partial_t \nabla \left( \Lq\W\right)$ satisfies the same bound.  Indeed,
\begin{align*}
    \left\| \partial_t \nabla \left(\Lq\W\right) \right\|_{C^0} &= \left\| \Dtq \left(\Lq \nabla\W\right) \right\|_{C^0} + \left\| \grad^\perp\Psi_q \cdot \grad \nabla\W \right\|_{C^0}\\
    &\leq \muqone \dqone^\frac{1}{2} + \dqone^\frac{1}{2} \laqone\\
    &\lesssim \dqone^\frac{1}{2} \laqone.
\end{align*}
Interpolation then shows that 
\begin{align*}
    \left\| \nabla \left(\Lq\W\right) \right\|_{C^\zeta_{x,t}} &\leq \left\| \nabla \left(\Lq\W\right) \right\|^{1-\zeta}_{C^0_{x,t}} \left\| \nabla \left(\Lq\W\right) \right\|^{\zeta}_{C^1_{x,t}}\\
    &\lesssim \left( \dqone^\frac{1}{2} \right)^{1-\zeta} \left( \dqone^\frac{1}{2} \right)^\zeta \laqone^\zeta\\
    &= \dqone^\frac{1}{2} \laqone^\zeta\\
    &= a^{\left(-\frac{1}{2}+c\zeta\right)b^{q+1}}.
    \end{align*}
By the assumption that $c>\frac{5}{2}$, we have that $-\frac{1}{2}+c\zeta$ is negative provided that $\zeta < \frac{1}{2c}<\frac{1}{5}$.  Then $\nabla\Psi_q$ is a convergence sequence in $C^\zeta_{t,x}$.  The bounds on the pressure follow immediately from \eqref{inductivecurl} and interpolation.
\end{proof}

\begin{proof}[Proof of \cref{maintheorem2d}]
Given that the extra assumption of \cref{eulerinductiveproposition} is satisfied at each stage $q$, every subsequent perturbation $\nabla\W$ can be taken to have zero third component, producing a solution to 2D Euler as desired after repeating the steps of the previous proof.
\end{proof}

\section{Appendix}\label{Appendix}
Here we collect several types of estimates which shall be necessary throughout the construction. All have become essentially standard in recent convex integration schemes. We begin with the following estimates for solutions to transport equations.  For a proof, we refer the reader to \cite{bdlisj15}.
\begin{lemma}[Transport Estimates]\label{transport}
Consider the transport equation
$$ \partial_t f + u \cdot \nabla f = g, \qquad \qquad f|_{t_0} = f_0 $$
where $f,g:\mathbb{T}^n\rightarrow \mathbb{R}$ and $u:\mathbb{T}^n\rightarrow \mathbb{R}^n$ are smooth functions.  Let $\Phi$ be the inverse of the flow $X$ of $u$ defined by 
$$ \frac{d}{dt}X = u(X,t) , \qquad X(x,t_0)=x.$$
Then the following hold:    
\begin{enumerate}
    \item $\| f(t) \|_{C^0} \leq \|f_0\|_{C^0} + \int_{t_0}^t \|g(s)\|_{C^0} \,ds $
    \item $\| Df(t) \|_{C^0} \leq \| Df_0 \|_{C^0} e^{(t-t_0)\|Du\|_{C^0}} + \int_{t_0}^t e^{(t-s)\|Du\|_{C^0}} \|Dg(s)\|_{C^0} \,ds $
    \item For any $N\geq 2$, there exists a constant $C=C(N)$ such that
    \begin{align*}
        \qquad\quad \|D^N f(t)\|_{C^0} \leq  &\left( \|D^N f_0\|_{C^0} + C(t-t_0)\|D^n u\|_{C^0} \| Df \|_{C^0} \right) e^{C(t-t_0)\|Du\|_{C^0}}\\
        & \quad + \int_{t_0}^t e^{C(t-s)\|Du\|_{C^0}} \left( \|D^N g(s)\|_{C^0} + C(t-s) \|D^N u \|_{C^0} \|Dg(s)\|_{C^0} \right) \,ds
    \end{align*}
    \item $ \|D\Phi(t)-\Id \|_{C^0} \leq e^{(t-t_0)\|Du\|_{C^0}} - 1 \leq (t-t_0)\|Du\|_{C^0} e^{(t-t_0)\|Du\|_{C^0}} $
    \item For $N\geq 2$ and a constant $C=C(N)$, $$ \| D^N \Phi(t)\|_{C^0} \leq C(t-t_0)\| D^N u \|_{C^0} e^{C(t-t_0)\|Du\|_{C^0}} $$
\end{enumerate}
\end{lemma}

The following estimate controls the norms of compositions of functions, particularly the perturbation.
\begin{lemma}[Chain Rule]\label{chainrule}
Let $\Omega\subset \mathbb{R}^D$ $f:\Omega \rightarrow \mathbb{R}$, $g:\mathbb{R}^d \rightarrow \Omega$ be smooth functions.  Then for every integer $N\geq 1$, there is a constant $C=C(N,d,D)$ such that 
$$ \|D^N(f\circ g)\|_{C^0} \leq C\left( \|Df\|_{C^0}\|D^N g\|_{C^0} + \|Df\|_{C^{N-1}} \|g\|_{C^0}^{N-1} \|D^N g\|_{C^0} \right) $$
and 
$$ \|D^N(f\circ g)\|_{C^0} \leq C\left( \|Df\|_{C^0}\|D^N g\|_{C^0} + \|Df\|_{C^{N-1}}  \|D g\|_{C^0}^N \right). $$
\end{lemma}

We shall make use of the following commutator estimates. The estimate in \cref{commutatorconstantin} is essentially contained in \cite{con15}, although the version stated here is a slight alteration whose statement and proof can be found in \cite{bdlsv18}.  The commutator estimate \eqref{commutatorbsv1} for convolution operators localized in frequency can be found in \cite{iv15} or \cite{bsv16}.  The estimates \eqref{iteratedcommutator} and \eqref{commutatorbsv2} follow the methods of proof given in \cite{iv15} and \cite{bsv16}.

\begin{prop}\label{commutatorconstantin}
Let $\alpha\in(0,1)$ and $N\geq 0$.  Let $T_K$ be a $\mathbb{R}^n$-Calder\'{o}n-Zygmund operator with kernel $K$.  Let $b \in C^{N+1,\alpha}(\mathbb{T}^n)$ be a vector field and $f\in C^{N,\alpha}(\mathbb{T}^n)$. Then there exists a constant $C=C(\alpha,N,K)$ such that 
\begin{align*}
 \| \left[ T_K,b\cdot\nabla \right]f \|_{N+\alpha} \leq C \|b\|_{1+\alpha}\|f\|_{N+\alpha} + \|b\|_{N+1+\alpha}\|f\|_{\alpha}. 
 \end{align*}
\end{prop}

\begin{prop}
Let $s\in\mathbb{R}$, $\lambda\geq1$, and let $T_K$ be an order $s$ convolution operator localized at length scale $\lambda^{-1}$ whose action on smooth functions is given by convolution with a kernel $K$ satisfying the bounds
$$ \| |x|^a \nabla^b K(x) \|_{L^1(\mathbb{R}^n)} \leq C(a,b) \lambda^{b-a+s} $$
for all $0\leq a, |b|$.  Then the following hold.
\begin{enumerate}
\item For $f:\mathbb{T}^n \rightarrow \mathbb{C}$ a smooth function and $u: \mathbb{T}^n \rightarrow \mathbb{R}^n$ a smooth vector field with $\nabla\cdot u = 0$, we have
\begin{align}\label{commutatorbsv1}
 \| \left[ u\cdot\nabla, T_K \right]f \|_{C^0} \leq \lambda^s \|\nabla u\|_{C^0} \|f\|_{C^0} 
 \end{align}
\item For $f:\mathbb{T}^n \rightarrow \mathbb{C}$ a smooth function and $u: \mathbb{T}^n \rightarrow \mathbb{R}^n$ a smooth vector field with $\nabla\cdot u = 0$, the iterated commutator $\left[ \partial_t + u \cdot \grad , \left[ u \cdot \grad, T_K \right] \right](f)$ obeys the estimate
\begin{align}\label{iteratedcommutator}
\left\| \left[ \partial_t + u \cdot \grad , \left[ u \cdot \grad, T_K \right] \right](f) \right\| \lesssim \lambda^{s-1} \| u \|_{C^1}^2 \| f \|_{C^1} + \|f\|_{C^0} \left( \lambda^{s+1} \| \partial_t u + u \cdot \grad u \|_{C^0} + \lambda^s \|u\|_{C^1}^2  \right).
\end{align}
\item For $f,g:\mathbb{T}^n \rightarrow \mathbb{C}$ smooth functions, we have (for an implicit constant depending on $k$ as well)
\begin{align}\label{commutatorbsv2}
\| \left[ g,T_K \right]f \|_{C^k} \lesssim \lambda^{s-1} \sum_{0\leq j \leq k} \|\nabla g\|_{C^j} \|f\|_{C^{k-j}}.
\end{align}
\end{enumerate}
\end{prop}
\begin{proof}
The proof of (1) is contained in the appendix of \cite{bsv16}.  Moving on to the iterated commutator estimate of (2), we first write
\begin{align*}
    \left[ u \cdot \grad, T_K \right](f) &= u(x) \cdot \grad \int_{\mathbb{R}^3} K(y)f(x-y)\,dy - \int_{\mathbb{R}^3} K(y) u(x-y)\cdot \grad f(x-y)\,dy \\
    &= \int_{\mathbb{R}^3} f(x-y) \grad K(y) \cdot \left(u(x) - u(x-y)\right).
\end{align*}
Now expanding the iterated commutator, we have
\begin{align*}
    &\left[ \partial_t + u \cdot \grad, \left[ u\cdot \grad, T_K \right] \right](f) = \\ &\qquad \left( \partial_t + u(x) \cdot \grad \right)\left( \int_{\mathbb{R}^3}  f(x-y)\grad K(y)\cdot \left( u(x)-u(x-y) \right) \right)\,dy \\
    &\qquad \qquad - \int_{\mathbb{R}^3} \left( \partial_t f(x-y) + u(x-y) \cdot \grad f(x-y) \right) \grad K(y) \cdot \left( u(x) - u(x-y) \right)\,dy\\
    &\qquad = \int_{\mathbb{R}^3} \left( \left( u(x)-u(x-y) \right)\cdot \grad f(x-y) \right) \grad K(y) \cdot \left( u(x)-u(x-y) \right)\,dy \\
    &\qquad \qquad + \int_{\mathbb{R}^3} f(x-y) \grad K(y) \cdot \left( \partial_t u(x) + u(x)\cdot \grad u(x) - \partial_t u(x-y) - u(x)\cdot \grad u(x-y) \right)\,dy \\
    &\qquad =: I + II.
\end{align*}
Estimating $I$ first, we write
\begin{align*}
    I &\leq \int_{\mathbb{R}^3} |\grad K(y)| \| u \|_{C^1}^2 |y|^2 \| f \|_{C^1} \,dy \\
    &\leq \| u \|_{C^1}^2 \| f \|_{C^1} \lambda^{s-1}.
\end{align*}
Before estimating II, note that 
$$ \partial_t u(x-y) + u(x)\cdot \grad u(x-y) = \left( \left( \partial_t + u \cdot \grad \right)u\right)(x-y) + \left( \left( u(x)-u(x-y) \right)\cdot \grad u(x-y) \right).$$
Therefore, 
\begin{align*}
    II & 2\leq \int_{\mathbb{R}^3}  \| f \|_{C^0} |\grad K(y)| \left( \| \partial_t u + u \cdot \grad u \|_{C^0} + |y| \|u\|_{C^1}^2 \right) \,dy\\
    &\lesssim \|f\|_{C^0} \left( \lambda^{s+1} \| \partial_t u + u \cdot \grad u \|_{C^0} + \lambda^s \|u\|_{C^1}^2  \right).
\end{align*}
Combining the estimates gives the result.

To prove (3), we follow the idea from \cite{bsv16} and write that
\begin{align*}
    \left|\nabla^k \left( T_K(bf)(x) - b(x) T_K f(x) \right) \right| &= \left| \int_{\mathbb{R}^n} \nabla^k \left( (b(x) - b(x-y)) f(x-y) \right) K(y) \,dy \right|\\
    &= \left| \int_{\mathbb{R}^n} \nabla^k \left( \left( \int_0^1 \nabla b(x-sy) \,ds \right) \cdot y f(x-y) \right) K(y) \,dy \right|.
\end{align*}
Applying the Leibniz rule and using the integrability assumption on $K$ finishes the proof.
\end{proof}

\begin{subsection}{Infinitely Many Weak Solutions Sharing the Same Initial Data}\label{ss:infinitelymany}
In this subsection, we state and outline the adjustments necessary to prove that infinitely many weak solutions may arise out of a single smooth initial data. We present and outline the proof of the following theorem.
\begin{theorem}\label{nonuniqueness:appendix}
Let $\Psi_0\in C^\infty(\mathbb{T}^2\times[0,2\pi])$ be a given mean-zero, smooth initial datum with an associated unique classical solution $\nabla\Psi$ of 3D QG on the time interval $[0,T_0]$, with $T_0<T_{\textnormal{max}}$ where $T_{\textnormal{max}}$ is the maximal time of existence. Then given $\zeta\in(0,\frac{1}{5})$, there exist infinitely many weak solutions $\{\nabla\Psi^\theta\}_{\theta\in\Theta}$ each belonging to $C^\zeta([0,T_0]\times\mathbb{T}^2\times[0,2\pi])$ such that for all $\theta\in\Theta$ and all $t\in\left[0,\frac{T_0}{2}\right)$ and all $(x,y,z)\in\mathbb{T}^2\times[0,2\pi]$,
$$  \nabla\Psi^\theta(t,x,y,z) = \nabla \Psi(t,x,y,z) . $$
\end{theorem}

We now briefly outline the proof of such a theorem using the techniques used to prove \cref{maintheorem}. First, the existence of a unique classical solution on a time interval $[0,T_0]$ dependent on the initial datum follows from standard methods.  Now consider smooth functions $c_1(t,z):[0,T_0]\times[0,2\pi]\rightarrow \mathbb{R}$ and $c_2(t):[0,T_0]\rightarrow \mathbb{R}$ which will later be chosen to satisfy several criteria.  Considering that $\Psi$ is a solution to 3D QG and repeating the calculation from \cref{stationarysolutions} which shows that functions of $z$ factor out of the nonlinear term, we have that $\Psi(t,x,y,z)c_1(t,z) + c_2(t)$ is a solution to
\begin{align*}
    \partial_t\nabla(\Psi c_1 + c_2 ) &+ \grad^\perp \left( \Psi c_1 + c_2 \right)\cdot\grad\nabla\left(\Psi c_1 + c_2 \right) = c_1 \left( \partial_t \nabla \Psi + \grad^\perp \Psi \cdot \grad \nabla \Psi \right)\\
    &\qquad\qquad + \partial_t c_1 \nabla \Psi + \partial_t \Psi \nabla c_1 + \partial_t \partial_z c_1 \Psi + \left( c_1^2 - c_1 \right) \grad^\perp\Psi\cdot\grad\nabla\Psi\\
    &= c_1 \curl Q \\
    &\quad + \Pgradbar\left( \partial_t c_1 \nabla \Psi + \partial_t \Psi \nabla c_1 + \partial_t \partial_z c \Psi + \left( c_1^2 - c_1 \right) \grad^\perp\Psi\cdot\grad\nabla\Psi \right)\\
    &\quad + \Pgradbarperp\left( \partial_t c_1 \nabla \Psi + \partial_t \Psi \nabla c_1 + \partial_t \partial_z c \Psi + \left( c_1^2 - c_1 \right) \grad^\perp\Psi\cdot\grad\nabla\Psi \right)\\
    &= \curl (c_1 Q) - \Pgradbar (\nabla c_1 \times Q) - \Pgradbarperp (\nabla c_1 \times Q) \\
    &\quad + \Pgradbar\left( \partial_t c_1 \nabla \Psi + \partial_t \Psi \nabla c_1 + \partial_t \partial_z c \Psi + \left( c_1^2 - c_1 \right) \grad^\perp\Psi\cdot\grad\nabla\Psi \right)\\
    &\quad + \Pgradbarperp\left( \partial_t c_1 \nabla \Psi + \partial_t \Psi \nabla c_1 + \partial_t \partial_z c \Psi + \left( c_1^2 - c_1 \right) \grad^\perp\Psi\cdot\grad\nabla\Psi \right).
\end{align*}

We now show that after choosing $c_1$ and $c_2$ carefully and making some simple observations, the convex integration procedure can be applied starting from $\Psi c_1 + c_2$, thus proving \cref{nonuniqueness:appendix}. We choose $c_1$ to be uniformly equal to one except on a compact subset of $[0,2\pi]\times(\frac{T_0}{2},T_0)$, with additional assumptions to follow later. We describe step-by-step how to verify the inductive assumptions for $\Psi c_1 + c_2$ at level $q=0$.
\begin{enumerate}
    \item Assumption \eqref{inductiveequation} follows after application of the inverse divergence operator to the $\Pgradbar$ terms, noticing that $\curl(c_1 Q)$ has vanishing third component at the boundary,  and noting that since $c_1$ is equal to one except on a compact set in $z$ and $t$, the $\Pgradbarperp$ terms have compact support in $z$ and vanish at $z=0,2\pi$, so that they may be absorbed into the curl.  We also choose $c_2$ so that $\Psi c_1+ c_2$ has mean zero in space for each time $t\in[0,T_0])$ 
    \item Assumption \eqref{inductivefrequencysupport} can be ensured by applying a Littlewood-Paley projector in $x$ and $y$ only to the equation satisfied by $\Psi c_1 + c_2$, which will produce a commutator stress after hitting the nonlinear terms.  Assuming that the projector acts as the identity on frequencies less than $\lambda_0$ and that $\lambda_0$ is sufficiently large, this commutator stress can be made arbitrarily small uniformly in time by the smoothness of $\Psi$
    \item Assumption \eqref{inductivespatialsupport} is satisfied for the matrix field $\Mdot_q$ due to $c_1$ equalling one except on a compact set in $z$ and $t$ and the fact that the $\Pgradbar$ and inverse divergence operators involve only convolution in $x$ and $y$. While $\Psi c_1 + c_2$ and the new curl will not be compactly supported in $z$, these assumptions are not strictly necessary to the convex integration scheme (the compact support in $z$ of $\Mdot_q$ was required to ensure that we can cancel it by adding perturbations compactly supported in $z$)
    \item Assumption \eqref{inductivevelocity} can be ensured by a sufficiently large choice of $\lambda_0$, the Littlewood-Paley projector in $x$ and $y$ acting as the identity on frequencies $\leq \lambda_0$, and a mollification in $z$ at sufficiently fine spatial scale inversely proportional to $\lambda_0$
    \item Assumption \eqref{inductiveerror} can be ensured by choosing the derivatives in $t$ and $z$ of $c_1$ to be small 
    \item Assumption \eqref{inductivetransport} can be ensured by a sufficiently large choice of $\lambda_0$
    \item Assumption \eqref{inductivecurl} can again be ensured by a large choice of $\lambda_0$
    \item Assumptions \eqref{inductiveenergyprofileone} and \eqref{inductiveenergyprofiletwo} can be ensured by choosing the energy profile $e(t)$ to be constant on $[0,\frac{T_0}{2})$ and to be slightly larger on the support of $c_1$ so that there is room for the addition of subsequent perturbations with non-zero energy  
\end{enumerate}
Thus we can construct a weak solution verifying the inductive assumptions of the convex integration procedure at level $q=0$.  Since there is no error on the time interval $[0,\frac{T_0}{2})$, the final weak solution constructed will agree with $\nabla\Psi$ for those times. The H\"{o}lder regularity follows as well.  Producing infinitely many such solutions follows from translating the support of the function $c_1$ in time and space.
\end{subsection}

\bibliography{references}

\begin{thebibliography}{10}

\bibitem{bbv19}
Rajendra {Beekie}, Tristan {Buckmaster}, and Vlad {Vicol}.
\newblock {Weak solutions of ideal MHD which do not conserve magnetic
  helicity}.
\newblock {\em arXiv e-prints}, page arXiv:1907.10436, Jul 2019.

\bibitem{bogovskii}
Mikhail Bogovskii.
\newblock Bogovskiĭ, m. e. solutions of some problems of vector analysis,
  associated with the operators ${\rm div}$ and ${\rm grad}$. (russian) theory
  of cubature formulas and the application of functional analysis to problems
  of mathematical physics, pp. 5--40, 149, trudy sem. s. l. soboleva, no. 1,
  1980, akad. nauk sssr sibirsk. otdel., inst. mat., novosibirsk, 1980.
  mr0631691 (82m:26014).
\newblock {\em Trudy Seminara S. L. Soboleva}, pages 5--40, 01 1980.

\bibitem{bb}
A~Bourgeois and J~Beale.
\newblock Validity of the quasigeostrophic model for large-scale flow in the
  atmosphere and ocean.
\newblock {\em SIAM Journal on Mathematical Analysis}, 25(4):1023--1068, 1994.

\bibitem{BCV2018}
T.~{Buckmaster}, M.~{Colombo}, and V.~{VIcol}.
\newblock {Wild solutions of the Navier-Stokes equations whose singular sets in
  time have Hausdorff dimension strictly less than 1}.
\newblock {\em ArXiv e-prints}, September 2018.

\bibitem{bsv16}
T.~{Buckmaster}, S.~{Shkoller}, and V.~{Vicol}.
\newblock {Nonuniqueness of weak solutions to the SQG equation}.
\newblock {\em ArXiv e-prints}, October 2016.

\bibitem{buc15}
Tristan Buckmaster.
\newblock Onsager's conjecture almost everywhere in time.
\newblock {\em Comm. Math. Phys.}, 333(3):1175--1198, 2015.

\bibitem{bdlisj15}
Tristan Buckmaster, Camillo De~Lellis, Philip Isett, and L\'aszl\'o
  Sz\'ekelyhidi, Jr.
\newblock Anomalous dissipation for {$1/5$}-{H}\"older {E}uler flows.
\newblock {\em Ann. of Math. (2)}, 182(1):127--172, 2015.

\bibitem{bdls16}
Tristan Buckmaster, Camillo De~Lellis, and L\'aszl\'o Sz\'ekelyhidi, Jr.
\newblock Dissipative {E}uler flows with {O}nsager-critical spatial regularity.
\newblock {\em Comm. Pure Appl. Math.}, 69(9):1613--1670, 2016.

\bibitem{bdlsv18}
Tristan Buckmaster, Camillo~De Lellis, L{\'{a}}szl{\'{o}} Sz{\'{e}}kelyhidi,
  and Vlad Vicol.
\newblock Onsager's conjecture for admissible weak solutions.
\newblock {\em Communications on Pure and Applied Mathematics}, jul 2018.

\bibitem{bv172}
Tristan Buckmaster and Vlad Vicol.
\newblock Nonuniqueness of weak solutions to the navier-stokes equation.
\newblock 2017.

\bibitem{bv19}
Tristan {Buckmaster} and Vlad {Vicol}.
\newblock {Convex integration and phenomenologies in turbulence}.
\newblock {\em arXiv e-prints}, page arXiv:1901.09023, Jan 2019.

\bibitem{cv}
L.~Caffarelli and A.~Vasseur.
\newblock Drift diffusion equations with fractional diffusion and the
  quasi-geostrophic equation.
\newblock {\em Ann. Math.}, 171(3):1903--1930, Apr 2010.

\bibitem{cl19}
Alexey {Cheskidov} and Xiaoyutao {Luo}.
\newblock {Stationary and discontinuous weak solutions of the Navier-Stokes
  equations}.
\newblock {\em arXiv e-prints}, page arXiv:1901.07485, Jan 2019.

\bibitem{cdlsj12}
A.~{Choffrut}, C.~{De Lellis}, and L.~{Sz{\'e}kelyhidi}, Jr.
\newblock {Dissipative continuous Euler flows in two and three dimensions}.
\newblock {\em ArXiv e-prints}, May 2012.

\bibitem{cs14}
A.~Choffrut and L.~Sz\'{e}kelyhidi, Jr.
\newblock Weak solutions to the stationary incompressible {E}uler equations.
\newblock {\em SIAM J. Math. Anal.}, 46(6):4060--4074, 2014.

\bibitem{Choffrut2013}
Antoine Choffrut.
\newblock h-principles for the incompressible euler equations.
\newblock {\em Archive for Rational Mechanics and Analysis}, 210(1):133--163,
  Oct 2013.

\bibitem{Colombo2018}
Maria Colombo, Camillo~De Lellis, and Luigi~De Rosa.
\newblock Ill-posedness of leray solutions for the hypodissipative
  navier{\textendash}stokes equations.
\newblock {\em Communications in Mathematical Physics}, 362(2):659--688, June
  2018.

\bibitem{cin18}
P.~{Constantin}, M.~{Ignatova}, and H.~{Nguyen}.
\newblock {Inviscid limit for SQG in bounded domains}.
\newblock {\em ArXiv e-prints}, June 2018.

\bibitem{cmt}
P.~Constantin, A.J. Majda, and E.~Tabak.
\newblock Singular front formation in a model for quasigeostrophic flow.
\newblock {\em Nonlinearity}, 7(6):1495--1533, Nov 1994.

\bibitem{cvicol}
P.~Constantin and V.~Vicol.
\newblock Nonlinear maximum principles for dissipative linear nonlocal
  operators and applications.
\newblock {\em Geometric and Functional Analysis}, 22(5):1289--1321, 2012.

\bibitem{con15}
Peter Constantin.
\newblock Lagrangian-{E}ulerian methods for uniqueness in hydrodynamic systems.
\newblock {\em Adv. Math.}, 278:67--102, 2015.

\bibitem{ci2}
Peter Constantin and Mihaela Ignatova.
\newblock Critical {SQG} in bounded domains.
\newblock {\em Annals of PDE}, 2(2):8, Nov 2016.

\bibitem{ci}
Peter Constantin and Mihaela Ignatova.
\newblock Remarks on the fractional laplacian with dirichlet boundary
  conditions and applications.
\newblock {\em International Mathematics Research Notices}, 2017(6):1653--1673,
  2017.

\bibitem{cn}
Peter Constantin and Huy~Quang Nguyen.
\newblock Global weak solutions for sqg in bounded domains.
\newblock {\em Communications on Pure and Applied Mathematics}, 2017.

\bibitem{cn2}
Peter Constantin and Huy~Quang Nguyen.
\newblock Local and global strong solutions for {SQG} in bounded domains.
\newblock {\em Physica D: Nonlinear Phenomena}, sep 2017.

\bibitem{cvicoltarfulea}
Peter Constantin, Andrei Tarfulea, and Vlad Vicol.
\newblock Long time dynamics of forced critical {SQG}.
\newblock {\em Comm. Math. Phys.}, 335(1):93--141, 2015.

\bibitem{dai}
Mimi {Dai}.
\newblock {Non-uniqueness of Leray-Hopf weak solutions of the 3D Hall-MHD
  system}.
\newblock {\em arXiv e-prints}, page arXiv:1812.11311, Dec 2018.

\bibitem{Daneri2014}
S.~Daneri.
\newblock Cauchy problem for dissipative h\"{o}lder solutions to the
  incompressible euler equations.
\newblock {\em Communications in Mathematical Physics}, 329(2):745--786, March
  2014.

\bibitem{Daneri2017}
Sara Daneri and L{\'{a}}szl{\'{o}} Sz{\'{e}}kelyhidi.
\newblock Non-uniqueness and h-principle for h\"{o}lder-continuous weak
  solutions of the euler equations.
\newblock {\em Archive for Rational Mechanics and Analysis}, 224(2):471--514,
  February 2017.

\bibitem{dsj17}
Sara Daneri and L\'aszl\'o Sz\'ekelyhidi, Jr.
\newblock Non-uniqueness and h-principle for {H}\"older-continuous weak
  solutions of the {E}uler equations.
\newblock {\em Arch. Ration. Mech. Anal.}, 224(2):471--514, 2017.

\bibitem{dls19}
camillo {De Lellis} and Jr~{Sz{\'e}kelyhidi}, L{\'a}szl{\'o}.
\newblock {On turbulence and geometry: from Nash to Onsager}.
\newblock {\em arXiv e-prints}, page arXiv:1901.02318, Jan 2019.

\bibitem{deLellis2008}
Camillo de~Lellis and L{\'{a}}szl{\'{o}} Sz{\'{e}}kelyhidi.
\newblock On admissibility criteria for weak solutions of the euler equations.
\newblock {\em Archive for Rational Mechanics and Analysis}, 195(1):225--260,
  December 2008.

\bibitem{dls13}
Camillo De~Lellis and L\'aszl\'o Sz\'ekelyhidi, Jr.
\newblock Dissipative continuous {E}uler flows.
\newblock {\em Invent. Math.}, 193(2):377--407, 2013.

\bibitem{dg}
B.~Desjardins and E.~Grenier.
\newblock Derivation of quasi-geostrophic potential vorticity equations.
\newblock {\em Adv. Differential Equations}, 3(5):715--752, 1998.

\bibitem{ise13b}
P.~{Isett}.
\newblock {Regularity in time along the coarse scale flow for the
  incompressible Euler equations}.
\newblock {\em ArXiv e-prints}, July 2013.

\bibitem{isettonsager}
P.~{Isett}.
\newblock {A Proof of Onsager's Conjecture}.
\newblock {\em ArXiv e-prints}, August 2016.

\bibitem{ise13a}
Philip Isett.
\newblock {\em Holder continuous {E}uler flows with compact support in time}.
\newblock ProQuest LLC, Ann Arbor, MI, 2013.
\newblock Thesis (Ph.D.)--Princeton University.

\bibitem{isett17}
Philip Isett.
\newblock Nonuniqueness and existence of continuous, globally dissipative euler
  flows.
\newblock 2017.

\bibitem{isettoh2016}
Philip Isett and Sung-Jin Oh.
\newblock On {Nonperiodic} {Euler} {Flows} with {Hölder} {Regularity}.
\newblock {\em Archive for Rational Mechanics and Analysis}, 221(2):725--804,
  August 2016.

\bibitem{iv15}
Philip Isett and Vlad Vicol.
\newblock H\"{o}lder continuous solutions of active scalar equations.
\newblock {\em Ann. PDE}, 1(1):Art. 2, 77, 2015.

\bibitem{kn}
A.~Kiselev and F.~Nazarov.
\newblock Variation on a theme of {C}affarelli and {V}asseur.
\newblock {\em Journal of Mathematical Sciences}, 166(1):31--39, Mar 2010.

\bibitem{knv}
A.~Kiselev, F.~Nazarov, and A.~Volberg.
\newblock Global well-posedness for the critical 2d dissipative
  quasi-geostrophic equation.
\newblock {\em Inventiones mathematicae}, 167(3):445--453, 2007.

\bibitem{dls14}
Camillo~De Lellis and L{\'{a}}szl{\'{o}}~Sz{\'{e}}kelyhidi Jr.
\newblock Dissipative euler flows and onsager's conjecture.
\newblock {\em Journal of the European Mathematical Society}, 16(7):1467--1505,
  2014.

\bibitem{delellisszek1}
Camillo~De Lellis and L{\'{a}}szl{\'{o}} Sz{\'{e}}kelyhidi.
\newblock The euler equations as a differential inclusion.
\newblock {\em Annals of Mathematics}, 170(3):1417--1436, nov 2009.

\bibitem{luo18}
X~{Luo}.
\newblock {Stationary solutions and nonuniqueness of weak solutions for the
  Navier-Stokes equations in high dimensions}.
\newblock {\em ArXiv e-prints}, July 2018.

\bibitem{Marchand}
F.~Marchand.
\newblock Existence and regularity of weak solutions to the quasi-geostrophic
  equations in the spaces {$L^p$} or {$\dot{H}^{-\frac{1}{2}}$}.
\newblock {\em Communications in Mathematical Physics}, 277(1):45--67, Jan
  2008.

\bibitem{boundeddomains}
M.~{Novack} and A.~{Vasseur}.
\newblock {The Inviscid 3D Quasi-Geostrophic System on Bounded Domains}.
\newblock {\em ArXiv e-prints}, June 2018.

\bibitem{novackweak}
M.~D. {Novack}.
\newblock {On the Weak Solutions to the 3D Inviscid Quasi-Geostrophic System}.
\newblock {\em ArXiv e-prints}, September 2017.

\bibitem{novack2019classical}
Matthew Novack and Alexis Vasseur.
\newblock Classical solutions for the 3d quasi-geostrophic system on a bounded
  domain, 2019.

\bibitem{novackvasseur}
Matthew~D. Novack and Alexis~F. Vasseur.
\newblock Global in time classical solutions to the 3d quasi-geostrophic system
  for large initial data.
\newblock {\em Communications in Mathematical Physics}, 358(1):237--267, nov
  2017.

\bibitem{pv}
M.~Puel and A.~Vasseur.
\newblock Global weak solutions to the inviscid 3{D} quasi-geostrophic
  equation.
\newblock {\em Communications in Mathematical Physics}, 339(3):1063--1082,
  2015.

\bibitem{nguyen17}
H.~{Quang Nguyen}.
\newblock {Global weak solutions for generalized SQG in bounded domains}.
\newblock {\em ArXiv e-prints}, April 2017.

\bibitem{Resnick}
S.~Resnick.
\newblock {\em Dynamical problems in non-linear advective partial differential
  equations}.
\newblock PhD thesis, University of Chicago, 1995.

\bibitem{DeRosa2018}
Luigi~De Rosa.
\newblock Infinitely many leray{\textendash}hopf solutions for the fractional
  navier{\textendash}stokes equations.
\newblock {\em Communications in Partial Differential Equations},
  44(4):335--365, December 2018.

\bibitem{scheffer}
Vladimir Scheffer.
\newblock An inviscid flow with compact support in space-time.
\newblock {\em J. Geom. Anal.}, 3(4):343--401, 1993.

\bibitem{shnirelman1}
A~Shnirelman.
\newblock On the nonuniqueness of weak solution of the {E}uler equation.
\newblock {\em Comm. Pure Appl. Math.}, 50(12):1261--1286, 1997.

\bibitem{Szkelyhidi2012}
L{\'{a}}szl{\'{o}} Sz{\'{e}}kelyhidi and Emil Wiedemann.
\newblock Young measures generated by ideal incompressible fluid flows.
\newblock {\em Archive for Rational Mechanics and Analysis}, 206(1):333--366,
  July 2012.

\end{thebibliography}
\bibliographystyle{plain}
\end{document}